\documentclass[sn-mathphys]{sn-jnl}

\jyear{2021}%

\theoremstyle{thmstyleone}%
\newtheorem{theorem}{Theorem}
\newtheorem{proposition}[theorem]{Proposition}%

\theoremstyle{thmstyletwo}%
\newtheorem{example}{Example}%
\newtheorem{remark}{Remark}%

\theoremstyle{thmstylethree}%
\newtheorem{definition}{Definition}%
\newcommand{\be}{\begin{eqnarray}}
\newcommand{\ee}{\end{eqnarray}}
\newcommand{\ce}{\begin{eqnarray*}}
\newcommand{\de}{\end{eqnarray*}}
\newtheorem{lemma}[theorem]{Lemma}
\newtheorem{corollary}[theorem]{Corollary}

\def\e{{\mathrm{e}}}

\def\a{\alpha}

\def\[{{\Big[}}
\def\]{{\Big]}}
\def\<{{\langle}}
\def\>{{\rangle}}
\def\({{\Big(}}
\def\){{\Big)}}

\def\bx{{\mathbf{x}}}

\def\dif{{\mathord{{\rm d}}}}

\def\min{{\mathord{{\rm min}}}}

\def\={&\!\!=\!\!&}
\def\bt{\begin{theorem}}
\def\et{\end{theorem}}
\def\bl{\begin{lemma}}
\def\el{\end{lemma}}
\def\br{\begin{remark}}
\def\er{\end{remark}}

\def\bd{\begin{definition}}
\def\ed{\end{definition}}
\def\bp{\begin{proposition}}
\def\ep{\end{proposition}}
\def\bc{\begin{corollary}}
\def\ec{\end{corollary}}
\def\bx{\begin{Examples}}
\def\ex{\end{Examples}}

\def\geq{\geqslant}
\def\leq{\leqslant}

\allowdisplaybreaks

\def\RR{\mathbb{R}}
\def\rmi{\mathrm{i}}
\def\lam{\lambda}
\def\al{\alpha}
\def\S{\mathbb{S}^{d-1}}
\def\EE{\mathbb{E}}
\def\ga{\gamma}

\begin{document}

\title[Multivariate stable approximation by Stein's method]{Multivariate stable approximation by Stein's method}


\author[1]{\fnm{Peng} \sur{Chen}}\email{chenpengmath@nuaa.edu.cn}
\equalcont{These authors contributed equally to this work.}

\author[2]{\fnm{Ivan} \sur{Nourdin}}\email{ivan.nourdin@uni.lu}
\equalcont{These authors contributed equally to this work.}

\author*[3,4]{\fnm{Lihu} \sur{Xu}}\email{lihuxu@um.edu.mo}
\equalcont{These authors contributed equally to this work.}

\author[5]{\fnm{Xiaochuan} \sur{Yang}}\email{xiaochuan.j.yang@gmail.com}
\equalcont{These authors contributed equally to this work.}

\affil[1]{\orgdiv{School of Mathematics}, \orgname{Nanjing University of Aeronautics and
Astronautics}, \orgaddress{\street{29 Jiangjun Avenue, Jiangning District}, \city{Nanjing}, \postcode{211106}, \state{Jiangsu}, \country{China}}}

\affil[2]{\orgdiv{Unit\'{e} de Recherche en Math\'{e}matiques}, \orgname{Universit\'{e} du Luxembourg}, \orgaddress{\street{Maison du Nombre, 6 avenue de la Fonte, L-4364 Esch-sur-Alzette}, \country{Grand Duchy of Luxembourg}}}

\affil*[3]{\orgdiv{Department of Mathematics, Faculty of Science and Technology}, \orgname{University of Macau}, \orgaddress{\street{Avenida da Universidade Taipa}, \city{Macau S.A.R.}, \country{China}}}

\affil[4]{\orgname{Zhuhai UM Science \& Technology Research Institute}, \orgaddress{\city{Zhuhai}, \state{Guangdong}, \country{China}}}

\affil[5]{\orgdiv{Department of Mathematics}, \orgname{Brunel University}, \orgaddress{\street{UB8 3PH}, \city{London}, \country{United Kingdom}}}


\abstract{By a delicate analysis for the Stein's equation associated to the $\alpha$-stable law approximation with
$\alpha \in (0,2)$, we prove a quantitative stable central limit theorem in Wasserstein type distance, which generalizes the results in the series of work \cite{Xu19, CNX19+,CNXYZ19+} from the univariate case to the multiple variate case. From an explicit computation for Pareto's distribution, we see that the rate of our approximation is sharp. The analysis of the Stein's equation is new and has independent interest.}




\maketitle

\section{Introduction}

This paper is concerned with the multivariate stable approximation by Stein's method. A probability measure $\pi$ on $\RR^d$ with $d\ge 2$ is \emph{strictly stable} if, for any $a>0$, there is $b>0$ such that
$\widehat{\pi}(\lambda)^a = \widehat{\pi}(b\lambda), \lambda\in\mathbb{R}^{d}$,
where $\widehat{\pi}$ is the characteristic function of $\pi$. For strictly stable measures, $a$ and $b$ have to satisfy the relation $b=a^{1/\al}$ where $\al\in(0,2)$ is  the \emph{stability parameter}.   A strictly $\al$-stable measure $\pi$ is characterised by a finite non-zero \emph{spectral measure} $\nu$ on the sphere $\mathbb S^{d-1}$ and, in and only in the  case $\al=1$, a vector $\ga\in\RR^d$, see \cite[Remark 14.6]{S}. Our working assumption in the case  $\al=1$ is that $\gamma=0$ and the center of mass of $\nu$ vanishes, namely, $\int_{\S}\theta \nu(d\theta)=0$. The first condition is artificial and the second is equivalent to strict stability of $\pi$ in the case $\al=1$. Under the condition $\ga=0$, we have a unified representation for any $\al\in (0,2)$,
\begin{align}\label{e:fourierpi}
\widehat{\pi}(\lambda)=\exp\left[ -\int_{\mathbb S^{d-1}}\nu(d\theta)\int_0^\infty (e^{\rmi\langle \lam, r\theta\rangle }-1-\rmi\langle \lam, r\theta\rangle k_\al(r) )  \frac{dr}{r^{1+\al}}\right], \quad \lambda\in\mathbb{R}^{d},
\end{align}
where $k_\al(r)= \mathbf{1}_{\al=1, r\in (0,1]} + \mathbf{1}_{\al\in(1,2)}$. The family of strictly stable laws is therefore as rich as the family of finite measures on $\S$. From now on, let $\psi$ denote the integral in the exponent of \eqref{e:fourierpi} and call it the \emph{characteristic exponent}.


\smallskip
The spectral measure $\nu$ plays a crucial role in the study of multivariate stable laws. The distributional properties of $\pi$ change dramatically from one type of $\nu$ to another.   For instance, if $\nu$ is the uniform probability measure on $\S$, then $\psi(\lam)=\sigma \vert \lam\vert ^\al$ with $\sigma>0$ so that $\pi$ is rotationally invariant. Hereafter, $\vert a\vert $ denotes the Euclidean norm of $a\in\RR^d$.   Another example is when $\nu = \sum_{i=1}^d \delta_{e_i}+\delta_{-e_i}$ where $\delta_.$ denotes the Dirac mass at some point and $\{e_i,1\le i\le d\}$ is the canonical basis of $\RR^d$, then $\psi(\lam) = \sum_{i=1}^d \sigma_i \vert \lam_i\vert ^\al$ for some $\sigma_i>0$ so that the marginal distributions of $\pi$ are independent one-dimensional symmetric stable laws. A third type of example is when $\nu$ is a fractal measure on $\S$, then $\pi$ can be wildly anisotropic with correlated marginals.  In this paper, we shall consider not only each of the aforementioned types of $\nu$, but also mixtures of these types.

\medskip
To assess convergence rates of a sequence towards a multivariate stable law, we use Stein's method -- a vast range of ideas and tools that allow one to study the proximity between a probability measure and a target distribution. The scope of the method has been considerably extended since Stein \cite{Stein86} proposed his elegant approach for normal approximation. In particular, Barbour \cite{Barbour88} devised the generator approach which is applicable to target distributions that can be realized as the stationary distribution of a "nice" Markov process. Barbour's approach is the one adopted in this paper and it takes the following steps. First, one constructs a Markov process $(X_t)_{t\ge 0}$ with infinitesimal generator $\mathcal A$ that converges in distribution to $\pi$ as $t\to\infty$ for any initial condition $X_0=x\in\RR^d$. Second, one considers Stein's equation (or Poisson equation in the PDE literature)
\begin{equation}\label{e:SE}
\mathcal A f(x) = h(x)-\pi(h)
\end{equation}
with $h\in L^1(\pi)$ and $\pi(h):=\int h(x)\pi(dx)$. By exploiting properties of the transition semigroup $(Q_t)_{t\ge 0}$ determined by $\mathcal A$, in particular $Q_0 h=h, Q_\infty h=\int h(x)\pi(dx)$ and the relation $\frac{d}{dt} Q_t = \mathcal A Q_t$,  one argues that
\begin{align}\label{e:f_h}
f_h(x):= - \int_0^\infty Q_t \big(h(x)-\pi(h)\big) dt
\end{align}
is in the domain of $\mathcal A$ and solves \eqref{e:SE}. Third, one uses the integral form \eqref{e:f_h} and properties of $(Q_t)_{t\ge 0}$ to derive regularity estimates for the solution \eqref{e:f_h}. To see why these steps lead to an upper bound for the distance between an arbitrary distribution and $\pi$,  let $Z$ denote a strictly stable random vector with distribution $\pi$, for any $\RR^d$-valued random vector $F$, one has
\begin{align*}
\EE[h(F)] - \EE[h(Z)] = \EE[\mathcal A f_h(F)].
\end{align*}
Ranging $h$ in a class of functions that is large enough to guarantee convergence in distribution, and using the regularity estimates of \eqref{e:f_h} obtained in the third step, together with the explicit form of $\mathcal A$, one would obtain an upper bound for a certain distance between $F$ and $Z$.

\medskip
Carrying out rigorously each of the aforementioned steps and claims in the context of stable approximation is a non-trivial task. In dimension one, Xu \cite{Xu19} considered the case of symmetric $\al$-stable law with $\al>1$. The approach of \cite{Xu19} was then generalized in \cite{CNX19+} to asymmetric $\al$-stable law with $\al>1$, and in \cite{AH19} to a class of infinitely divisible distributions with finite first moment. Later, Chen \textit{et al.} \cite{CNXYZ19+} considered non-integrable $\al$-stable approximation ($\alpha\le 1$). In higher dimensions, Arras and Houdr\'e \cite[Theorem 4.2]{AH19+} executed the aforementioned second step (construction of the solution to  Stein's equation) for a class of self-decomposable distributions which includes multivariate stable laws.  However, regularity estimates of the solution are obtained only when test functions have their 0-th, first and second partial derivatives bounded by 1. Therefore, the results in \cite{AH19+} cannot be used to derive bounds for multivariate stable approximation in Wasserstein(-type) distance that we address in this paper.   

\medskip
The main contribution of this paper is  the regularity estimates for the solution to Stein's equation in the context of multivariate stable approximation and \emph{Lipschitz(-type) test functions}, which in turn allows to obtain Wasserstein bounds. Such bounds were not available in previous work.  Our approach relies on delicate density estimates of multivariate strictly stable laws. Recent advances on heat kernel estimates of anisotropic non-local operators e.g. \cite{W07, ChZh16, C-Z3,B-S-3} allow us to handle a diversified class of spectral measures.  Since real life high dimensional data are often anisotropic (see \cite{MBB99,ER07,DWZ20} and the references therein), the rich class of spectral measures that we consider would widen the applicability of our results.  In terms of application, we provide the rate of convergence for the classical multivariate stable limit theorem.

\medskip

The rest of this paper is organized as follows.  After introducing the Markov process  converging to $\pi$, we construct a solution to Stein's equation (Proposition \ref{p:solution}), present the regularity estimates for the solution (Theorem \ref{t:reg}) and obtain Wasserstein bounds for multivariate stable approximation (Theorem \ref{t:bound}).   Theorem \ref{t:reg} is proved in Section \ref{s:reg} and Theorem \ref{t:bound} is proved in Section \ref{s:bound}.  Example is given in Section \ref{s:example}.

\section{Preliminaries and statement of the main results}

\subsection{Ornstein-Uhlenbeck type processes}\label{s:OU}
 The Markov process we construct in the first step of Barbour's program is the so-called Ornstein-Uhlenbeck type process which is a simple stochastic differential equation (SDE) driven by stable L\'evy processes.  We refer the reader to  Applebaum \cite{A} for background on stochastic calculus of L\'evy processes, and Sato \cite{S} for general facts about L\'evy processes.

Let $(Z_t)_{t\ge 0}$ be a strictly stable L\'evy process, a process with independent and stationary increments having marginal $Z_1$ distributed as $\pi$, given by \eqref{e:fourierpi}.  Consider the SDE
\begin{align}\label{e:OU}
\begin{cases} X_{t}= - \frac{1}{\al}\int_0^t X_s ds+ Z_{t} \\
X_0=x
\end{cases},
\end{align}
Such an equation can be solved explicitly
\begin{align}\label{e:OU_explicit}
X^x_t = xe^{-\frac t \alpha} + \int_0^t e^{-\frac {t-s} \alpha} dZ_s,
\end{align}
see \cite[p.105]{S}, and provides an interpolation between any Dirac mass and $\pi$. This follows from the fact that $(X^x_t)_{t\ge 0}$ is a scaled and time-changed L\'evy process, i.e.
\begin{align}\label{e:time-change}
X^x_t \overset{d}= xe^{-\frac t \alpha} + e^{-\frac t \alpha} Z_{e^t-1} \overset{d}= xe^{-\frac t \alpha} +  Z_{1-e^{-t}},
\end{align}
see \cite[Section 2.3]{CNXYZ19+}, where $\overset{d}=$ denotes equality in distribution. For the second equality we have used the self-similarity of the process $(Z_t)_{t\ge 0}$, namely $Z_{ct}\overset{d}= c^{1/\al}Z_t$ in distribution for any $c,t>0$.  One sees that as $t\to\infty$, $X^x_t$ converges in distribution to $Z_1\sim \pi$. For another proof of the latter fact, one may check the condition of a general result \cite[Th. 17.5]{S} for self-decomposable distributions.

An application of It\^o's formula for semimartingales with jumps to $(X^x_t)_{t\ge 0}$ shows  that (see \cite[Chapter 6]{A} for details) the generator of $X$ is
\begin{align}\label{e:operator_OU}
\mathcal A^{\al,\nu} f(x):= \mathcal L^{\al,\nu} f(x)  - \frac{1}{\al} \langle x, \nabla f(x)\rangle,
\end{align}
where recalling the definition of $k_\al(r)$ \eqref{e:fourierpi},
\begin{align}\label{representation}
\mathcal L^{\al,\nu} f(x) = d_\al    \int_{\mathbb S^{d-1}} \nu(d\theta)  \int_0^\infty ( f(x+ r\theta) - f(x) - k_\al(r) \langle r\theta, \nabla f(x) \rangle  )  \frac{dr}{r^{1+\al}}.
\end{align}
Here $d_{\alpha}^{-1}=\int_{0}^{\infty} (1-\cos y )y^{-1-\al}dy =\al^{-1}\Gamma(1-\alpha)\cos\frac{\pi\alpha}{2}$, $\al\in(0,2)\setminus\{1\}$ with $d_1= \lim_{\al\to 1} d_\al= 2/\pi$,  and $\nu$ is normalized to have total mass 1.

\subsection{Solving Stein's equation}

Now one can write down Stein's equation associated with the multivariate stable distribution $\pi$ as follows
\begin{align}\label{e:steinsequation}
\mathcal A^{\al,\nu} f(x) = h(x) - \pi(h),
\end{align}
where $h\in L^1(\pi)$.  In view of obtaining bounds in Wasserstein distance, we consider $h$ belonging to the space $\mathrm{Lip}_1$ of Lipschitz continuous functions with Lipschitz constant at most 1. It is standard that $\mathrm{Lip}_1\subset L^1(\pi)$ if $\al>1$, while $\mathrm{Lip}_1\cap L^1(\pi)$ is a strict subset of $\mathrm{Lip}_1$ if $\al\le 1$.  Whenever $\al\le 1$, we let $0<\beta<\al$ and consider $h\in \mathcal{H}_{\beta}:=\mathrm{Lip}_1 \cap \textrm{H\"ol}(\beta,1)$ where $\textrm{H\"ol}(\beta,1)$ is the class of H\"older continuous functions of order $\beta$ and H\"older constant at most 1. Namely $h\in \mathcal{H}_{\beta}$ means
\begin{align*}
\vert h(x)- h(y)\vert \le \vert x-y\vert \wedge \vert x-y\vert ^\beta.
\end{align*}
The Lipschitz continuity imposes smoothness and the H\"older condition imposes global growth rate of the function $h$ which is crucial to make sense of \eqref{e:steinsequation}  in the case $\al\le 1$ thanks to the simple inclusion $\textrm{H\"ol}(\beta,1)\subset L^1(\pi)$.

We construct a solution to  Stein's equation by using the process $(X^x_t)_{t\ge 0}$, as described in the introduction.  Denote by $p(t,x):=p_t(x)$ the density of the driving process $(Z_t)_{t\ge 0}$ in \eqref{e:OU}.  Write $p(x):=p_1(x)$.  By \eqref{e:time-change}, one sees that
\begin{equation}\label{density}
  q(t,x,y)=p_{1-e^{-t}}(y-e^{-t/\alpha}x)=s(t)^{-d/\alpha}p(s(t)^{-1/\alpha}(y-e^{-t/\alpha}x)),
\end{equation}
where $y\mapsto q(t,x,y)$ is the density of $X^x_t$,   $s(t)=1-e^{-t}$ and we used the self-similarity of $(Z_t)_{t\ge 0}$ in the second equality.

\begin{proposition}[Solution to Stein's equation]\label{p:solution}
Suppose $h\in \mathrm{Lip}_1$ if $\al>1$ and $h\in\mathcal H_\beta$ for some $\beta<\al$ if $\al\le 1$. Set
\begin{align}\label{solution}
f(x)&:=-\int_{0}^{\infty}\mathbb{E}\big[h\big(X_{t}^{x}\big)-  \pi(h)\big]dt,\nonumber \\
& = -\int_0^\infty \int
p_{1-e^{-t}}(y-e^{-\frac{t}{\alpha}}x)(h(y)-\pi(h))dydt\nonumber\\
&=-\int_{0}^{\infty}\int_{\mathbb{R}^{d}}p(y)\big[h((1-e^{-t})^{1/\alpha}y+e^{-t/\alpha}x)-h(y)\big]dydt.
\end{align}
Then $\mathcal A^{\al,\nu} f$ is well-defined and $f$ solves Stein's equation \eqref{e:steinsequation}.
\end{proposition}

The proof of this Proposition somewhat standard in view of recent developments \cite{Xu19, CNX19+,CNXYZ19+} on stable approximation with Lipschitz(-type) test functions, we give a proof in the Appendix for the sake of completeness.

\begin{remark}\label{rem}
\begin{itemize}
\item[(i)] Note that the last two identities follow from \eqref{density} and a change of variables.
We end this subsection by verifying that \eqref{solution} is well-defined. When $\alpha\in(1,2)$, this is obvious. When $\alpha\in(0,1]$,
since $h\in\mathcal H_\beta$, we have
\begin{align*}
&\big\vert h((1-e^{-t})^{1/\alpha}y+e^{-t/\alpha} x)- h(y)\big\vert  \\
&\le e^{-t/\alpha}\vert x\vert  + e^{-t\beta/\alpha} \vert x\vert ^\beta +  \vert y(1-(1-e^{-t})^{1/\alpha})\vert \wedge \vert y(1-(1-e^{-t})^{1/\alpha})\vert ^\beta,
\end{align*}
which is integrable with respect to $1_{t>0}dt\otimes p(y)dy$, as desired.
\item[(ii)] Throughout the paper, we often use the facts that $\mathbb{E}\vert Z_{1}\vert ^{\beta}<\infty$ for any $\beta\in(0,\alpha)$, where $Z_{t}$ is the strictly $\alpha$-stable L$\acute{e}$vy process. For the convenience of readers, we obtain the moment estimate of $Z_{1}$ in Lemma \ref{moment} in Appendix \ref{amoment}.
\end{itemize}
\end{remark}

\subsection{Probability metric}
In the case $\al>1$, we shall consider multivariate stable approximation in the classical Wasserstein distance given by
\begin{align*}
d_{W}(\mu_{1},\mu_{2})=\sup_{h\in\mathrm{Lip}_1}\big\vert \mu_{1}(h)-\mu_{2}(h)\big\vert .
\end{align*}
In the case $\al\le 1$, we shall consider Wasserstein-type distance given by
\begin{align*}
d_{W_{\beta}}(\mu_{1},\mu_{2})=\sup_{h\in\mathcal{H}_{\beta}}\big\vert \mu_{1}(h)-\mu_{2}(h)\big\vert ,\quad 0<\beta<\alpha.
\end{align*}

\subsection{Spectral measures}\label{s:zooofspectral}

Obtaining sharp density estimates for general multivariate stable law is very sensitive to the form of the spectral measure. The seminal work of Watanabe \cite{W07} seems to be the first which identifies the best, worse and a range of different rates of decay of stable densities in relation to the spectral measures.

For our purpose, we need not only density estimates but also  gradient and fractional derivative (the operator $\mathcal L^{\al,\nu}$) estimates of stable densities. Our setting is that of Bogdan \emph{et al.} \cite{B-S-3}. In fact,  \cite{B-S-3} addressed L\'evy-type operators with a jump kernel equivalent to the L\'evy measure of a general multivariate stable law. The condition on the spectral measure therein is that the corresponding L\'evy measure is a $\ga$-measure.  Let $\mu$ denote the L\'evy measure of a multivariate strictly stable distribution, i.e., for  measurable $A\subset\RR^d$,
\begin{align*}
\mu(A):= \int_{\S} \nu(d\theta) \int_0^\infty \mathbf{1}_{A}(r\theta) \frac{dr}{r^{1+\al}}.
\end{align*}

\begin{definition}

We say $\mu$ is a $\ga$-measure for some $\ga\ge 0$ if there exists a finite $c>0$ such that for all $x\in\S$ and $0<r<1/2$, one has
\begin{align*}
\mu(B(x,r)) \le c r^\ga.
\end{align*}
Here $B(x,r)$ is the Euclidean ball with center $x$ and radius $r>0$.
\end{definition}

It is easy to see that the L\'evy measure of stable laws are $\ga$-measures with $\ga\in [1,d]$.  The following examples show that this is a rather general setting.

\begin{example}\label{ex:abs}
Suppose that $\nu$ is absolutely continuous with respect to uniform measure on $\S$ with density $\kappa$ bounded from above and below. Then for measurable $A\subset\RR^d$,
\begin{align*}
\mu(A) = \int_{A} \frac{\kappa(x/\vert x\vert ) dx}{\vert x\vert ^{d+\al}}.
\end{align*}
This is the setting (restricting to the L\'evy case) of Chen and Zhang \cite{ChZh16, C-Z3} where the densities, their gradient and fractional derivatives are estimated. One readily checks that $\mu$ is a $d$-measure.
\end{example}

\begin{example}
Suppose that $\nu=\sum_{i=1}^k  a_i  \delta_{x_i}$  where $\delta_\cdot$ is a Dirac mass and $\{x_1,...,x_k\} \subset \S$. Then $\mu$ is a $d$-measure. In the particular case where $x_i$'s are the canonical basis and $k=d$, we have a multivariate stable law with independent marginals. The desired estimates follows from their counterparts in dimension 1.
\end{example}

\begin{example}
Let $1\le \ga<d$. Suppose that $\nu$ is supported on $E\subset \S$ which is Ahlfors regular of order $\ga-1$, namely, there exists $c$ such that
\begin{align*}
\mathcal H^{\ga-1}(B(x,r)\cap E) \le c r^{\ga-1}, \quad x\in E,
\end{align*}
where $\mathcal H^s$ is the $s$-dimensional Hausdorff measure. Then $\nu$ is $\ga$-measure whenever it is absolutely continuous with respect to $(\ga-1)$-dimensional Hausdorff measure with bounded density.  In \cite{B-S-3} only density and gradient estimates were obtained. Computing  fractional derivatives of stable densities requires a little more work, which we do  in this paper following the arguments of \cite{ChZh16, C-Z3}. It was observed in \cite{W07} that, subject to further lower bounds on the $\nu$-measure of balls, sharp two-sided estimates can be obtained for stable densities. Since we are only concerned with upper bounds, we do not need to impose these additional assumptions. For aspects of fractal measures, we refer to  \cite{Falconer}.
\end{example}

\begin{example}
Any linear combination of $\nu$ in the previous three examples is again a $\ga$-measure with an appropriate $\ga$.
\end{example}

We gather estimates about stable densities which are useful for our purpose.

\begin{lemma}\label{l:heatkernel}
Let $p(x)$ be the density of a multivariate strictly $\al$-stable law with characteristic exponent \eqref{e:fourierpi} for all $\al\in(0,2)$ and $\int_{\mathbb{S}^{d-1}}\theta\nu(d\theta)=0$ when $\a=1$. Suppose that the L\'evy measure $\mu$  is a $\ga$-measure with $\ga\in [1,d]$ satisfying $\ga>d-\al$. Suppose also that $\nu$ is symmetric i.e. $\nu(B)=\nu(-B)$ for any $B\subset \S$.  Then, there exists a finite constant $C=C_{\al,d,\nu}>0$ such that for all $x,y\in\RR^d$
\begin{align}\label{est:p}
\vert p(x)\vert  \leq  \frac{C}{(1+\vert x\vert )^{\alpha+\gamma}},
\end{align}
\begin{align}\label{fg4}
\vert \nabla p(x)\vert \leq  \frac{C}{(1+\vert x\vert )^{\alpha+\gamma}},
\end{align}
\begin{align}\label{tg4}
\|  \nabla^{2} p(x)\|  _\mathrm{op}\leq  \frac{C}{(1+\vert x\vert )^{\alpha+\gamma}},
\end{align}
and
\begin{align}\label{eg4}
\vert p(x)-p(y)\vert \leq C(\vert x-y\vert \wedge1)\Big(\frac{1}{(1+\vert x\vert )^{\alpha+\gamma}}+\frac{1}{(1+\vert y\vert )^{\alpha+\gamma}}\Big),
\end{align}
where $\nabla^{2}$ is the Hessian and $\|  \cdot\|  _\mathrm{op}$ denotes the operator norm.
\end{lemma}

\begin{remark}\label{r:heatkernel}
\begin{enumerate}
\item[a)] In the setting of Example \ref{ex:abs},  Chen and Zhang \cite{C-Z3} obtained these estimates with $\ga=d$ without the symmetry assumption on the spectral measure, extending their earlier work \cite{ChZh16}. It is an open problem to obtain these general estimates in the setting of \cite{B-S-3} without the  symmetry condition.
\item[b)] Regularity estimates of solutions to Stein's equation and Wasserstein bounds in the sequel rely on this Lemma. In view of Item a), all the upcoming results extend to non-symmetric $\nu$ in the setting of Example \ref{ex:abs}.
\item[c)] Equations \eqref{est:p}-\eqref{tg4} were proved in \cite[Lemma 2.4]{B-S-3}.  Equation \eqref{eg4} follows from \eqref{fg4} by the mean value theorem.
\item[d)] In the above Lemma, if $\nu(d\theta)=\frac{1}{V(\mathbb{S}^{d-1})}$ (the rotationally invariant $\alpha$-stable L$\acute{e}$vy process), where $V(\mathbb{S}^{d-1})$ is the surface area of $\mathbb{S}^{d-1}$ and $V(\mathbb{S}^{d-1})=\frac{2\pi^{d/2}}{\Gamma(d/2)}$, then we can obtain that
    \begin{align*}
    p(x)\leq \frac{\max\Big\{2^{-d+1}\pi^{-\frac{d}{2}}\frac{\Gamma(d/\alpha)}{\alpha\Gamma(d/2)},\frac{\alpha2^{\alpha-1}\sin\frac{\alpha\pi}{2}\Gamma((d+\alpha)/2)
\Gamma(\alpha/2)}{\pi^{d/2+1}}\Big\}}{(1+\vert x\vert )^{\alpha+d}},
    \end{align*}
that is, the dependence of the constant $C$ on the dimension in \eqref{est:p} is exponential and similar conclusions can be made about the gradient \eqref{fg4} and Hessian matrix \eqref{tg4}. The specific details will be given in Lemma \ref{abounds} in Appendix \ref{amoment}. For the general case, the dependence of the constants on the dimension may need to analysis some infinite series, which is beyond the scope of this paper, we omit here.
\end{enumerate}
\end{remark}

\subsection{Main results}

\begin{theorem}[Regularity estimates for the solution]\label{t:reg}
 Let $f$ be given by  \eqref{solution} for $h\in\mathrm{Lip}_1$ in the case $\al>1$ and $h\in\mathcal H_\beta$ for some $\beta<\alpha$ in the case $\al\le 1$. Suppose that the assumption of Lemma \ref{l:heatkernel} holds.
\begin{itemize}
\item[i)] If $\al\in(1,2)$. Then there exists a finite constant $C>0$ depending on $\al,d,\nu$,
\begin{equation}\label{first order}
\|  \nabla f\|  _\infty \leq \alpha,
\end{equation}
\begin{equation}\label{second order}
\sup_{x\in\RR^d}\|  \nabla^{2}f(x)\|  _\mathrm{op} \leq C,
\end{equation}
where $\|  f\|  _\infty = \sup_{x\in\RR^d} \vert f(x)\vert $ is the $L^{\infty}$ norm. Further,
for all $x, y\in\mathbb{R}^{d}$
\begin{equation}\label{e:fracder}
\Big\vert \mathcal{L}^{\alpha,\nu} f(x)-\mathcal{L}^{\alpha,\nu}f(y)\Big\vert \leq\frac{2d_{\alpha}C }{\alpha(2-\alpha)(\alpha-1)}\vert x-y\vert ^{2-\alpha}.
\end{equation}

\item[ii)] If $\al=1$, then there exists finite $C>0$ depending on $\beta,d,\nu$ such that
\begin{align}\label{L9}
\|  \nabla f\|  _{\infty}\leq 1,
\end{align}
\begin{align}\label{L7}
\|  \mathcal{L}^{1,\nu}f\|  _{\infty}\leq C
\end{align}
and for any $x,y\in\RR^d$ with $\vert x-y\vert <1$ and $w\in\RR^d$,
\begin{align}\label{L3}
\vert \nabla f(x)-\nabla f(y)\vert \leq C \big(1-\log\vert x-y\vert \big)\vert x-y\vert .
\end{align}
\begin{align}\label{L4}
\vert f(x+w)-f(x)\vert \leq C \vert w\vert \wedge\vert w\vert ^{\beta},
\end{align}
\begin{align}\label{decay1}
\vert \langle x,\nabla f(x)\rangle\vert \leq C (1+\vert x\vert ^{\beta}).
\end{align}
\begin{align}\label{L8}
\vert \mathcal{L}^{1,\nu}f(x)-\mathcal{L}^{1,\nu}f(y)\vert \leq C \vert x-y\vert (1-\log\vert x-y\vert ).
\end{align}

\item[iii)] If $\al\in(0,1)$, then there exists finite $C>0$ depending on $\al,\beta,d,\nu$ such that
\begin{align}\label{firstb}
\|  \nabla f\|  _{\infty}\leq\alpha,
\end{align}
\begin{align}\label{L5}
\|  \mathcal{L}^{\alpha,\nu}f\|  _{\infty}\leq C,
\end{align}
\begin{align}\label{L1}
\vert \nabla f(x)-\nabla f(y)\vert  \leq C \vert x-y\vert ^\al
\end{align}
and for any $x,y\in\RR^d$, one has
\begin{align}\label{L2}
\vert f(x+y)-f(x)\vert \leq C\vert y\vert \wedge\vert y\vert ^{\beta},
\end{align}
\begin{align}\label{decay}
\vert \langle x,\nabla f(x)\rangle\vert \leq C (1+\vert x\vert ^{\beta}).
\end{align}
Further, for any $\eta\in (0,1)$, there exists finite $C>0$ depending on $\al,\beta,d,\nu,\eta$ such that for any $x,y\in\RR^d$,
\begin{align}\label{L6}
\vert \mathcal{L}^{\alpha,\nu}f(x)-\mathcal{L}^{\alpha,\nu}f(y)\vert \leq C \vert x-y\vert ^{\eta}.
\end{align}
\end{itemize}
\end{theorem}

In addition, we denote $\mathcal{F}_{\beta}$ is the class of functions $h:\mathbb{R}^{d}\rightarrow(\mathbb{R},d_{\beta})$ such that $\vert \nabla h(x)\vert \leq\frac{1}{1+\vert x\vert ^{1-\beta}}$. Then, we have the following proposition.

\begin{proposition}\label{prop}
Let $\alpha\in(0,\frac{1}{2}]$ and $f$ be given by \eqref{solution} for $h\in\mathcal{H}_{\beta}\cap\mathcal{F}_{\beta}$ with $\beta\in(0,\alpha)$. Suppose that the assumption of Lemma \ref{l:heatkernel} holds and denote $\tilde{\beta}:=\max\{\beta,d-\gamma\}\in(0,\alpha)$. Then, there exists a finite constant $C>0$ depending on $\al,\beta,d,\nu,\gamma$ such that for any $x\in\mathbb{R}^{d}$,
\begin{align*}
\vert \nabla f(x)\vert \leq C\left(1\wedge\vert x\vert ^{\tilde{\beta}-1}\right).
\end{align*}
\end{proposition}

\medskip

The second result is concerned with Wasserstein(-type) bounds for limit theorems with multivariate stable limit.

Let $n\in\mathbb{N}$ and let $\zeta_{n,1}, \zeta_{n,2}, \cdots, \zeta_{n,n}$ be a sequence of independent random vectors satisfying $\mathbb{E}\vert \zeta_{n,i}\vert ^{\beta}<\infty$ for any $\beta\in(0,\alpha)$ and $i=1,\cdots,n$.

Set
\begin{equation*}
S_{n}=
\begin{cases}
\zeta_{n,1}-\mathbb{E}\zeta_{n,1}+\cdots+\zeta_{n,n}-\mathbb{E}\zeta_{n,n},\quad &\alpha\in(1,2),\\
\zeta_{n,1}-\mathbb{E}\zeta_{n,1}{\bf 1}_{(0,1]}(\vert \zeta_{n,1}\vert )+\cdots+\zeta_{n,n}-\mathbb{E}\zeta_{n,n}{\bf 1}_{(0,1]}(\vert \zeta_{n,n}\vert ), &\alpha=1,\\
\zeta_{n,1}+\zeta_{n,2}+\cdots+\zeta_{n,n},\quad &\alpha\in(0,1)
\end{cases}
\end{equation*}
and
\begin{equation*}
S_{n}(i)=S_{n}-\zeta_{n,i}, \qquad 1\leq i\leq n.
\end{equation*}
Denote $l_{n}=\frac{\alpha}{d_{\alpha}}n$ and set $\eta_{n,i}=l_{n}^{1/\alpha}\zeta_{n,i}$, then we have the following Theorem.

\begin{theorem}[Wassertein bounds]\label{t:bound}
Suppose that the assumption of Lemma \ref{l:heatkernel} holds. Let $n\in\mathbb{N}$ and $\zeta_{n,i}$, $\eta_{n,i},$ $i=1,\cdots,n$ are defined as above. Denote the density function of $\eta_{n,i}$ by $p_{\eta_{n,i}}(r)dr\nu(d\theta)$ and when $\alpha\in(0,1]$, further assume that $p_{\eta_{n,i}}(r)$ is non-increasing,
\begin{itemize}
\item[(1)] When $\alpha\in(1,2)$, there exists a finite constant $C>0$ depending on $\al,d,\nu$ such that
\begin{align*}
&d_{W}\big(\mathcal{L}(S_{n}),\pi\big)\\
\leq& C\sum_{i=1}^{n}\left\{n^{-\frac{2}{\alpha}}\mathbb{E}\vert \eta_{n,i}\vert ^{2-\alpha}+n^{-\frac{2}{\alpha}}\big(\mathbb{E}\vert \eta_{n,i}\vert \big)^{2}\right.\\
&\left.\qquad+n^{-\frac{2}{\alpha}}\int_{0}^{l_{n}^{\frac{1}{\alpha}}}r^{2}\big\vert \frac{\alpha}{r^{\alpha+1}}-p_{\eta_{n,i}}(r)\big\vert dr
+n^{-\frac{1}{\alpha}}\int_{l_{n}^{\frac{1}{\alpha}}}^{\infty}\big\vert \frac{\alpha}{r^{\alpha}}-rp_{\eta_{n,i}}(r)dr\big\vert \right\}.
\end{align*}
\item[(2)] When $\alpha=1$, there exists a finite constant $C>0$ depending on $\beta,d,\nu$ such that
\begin{align*}
&d_{W_{\beta}}\big(\mathcal{L}(S_{n}),\mu\big)\\
\leq&C\sum_{i=1}^{n}\Big\{n^{-2}\int_{0}^{l_{n}}r\big(1-\log(l_{n}^{-1}r)\big)p_{\eta_{n,i}}(r)dr+n^{-\beta}\int_{l_{n}}^{\infty}r^{\beta}\big\vert d[\frac{1}{r}-rp_{\eta_{n,i}}(r)]\big\vert \\
&\quad+n^{-1}\int_{l_{n}}^{\infty}p_{\eta_{n,i}}(r)dr+n^{-2}\int_{0}^{l_{n}}r^{2}\big(1-\log(l_{n}^{-1}r)\big)\vert \frac{1}{r^{2}}dr-p_{\eta_{n,i}}(r)dr\vert \Big\}.
\end{align*}
\item[(3)] When $\alpha\in(\frac{1}{2},1)$, there exists a finite constant $C>0$ depending on $\al,\beta,d,\nu$ such that
\begin{align*}
&d_{W_{\beta}}\big(\mathcal{L}(S_{n}),\mu\big)\\
\leq&C\sum_{i=1}^{n}\Big\{n^{-\frac{3\alpha+1}{2\alpha}}\int_{0}^{l_{n}^{\frac{1}{\alpha}}}r^{\frac{\alpha+1}{2}}p_{\eta_{n,i}}(r)dr+n^{-\frac{\beta}{\alpha}}\int_{l_{n}^{\frac{1}{\alpha}}}^{\infty}r^{\beta}\big\vert d[\frac{\alpha}{r^{\alpha}}-rp_{\eta_{n,i}}(r)]\big\vert \\
&\quad+n^{-1}\int_{l_{n}^{\frac{1}{\alpha}}}^{\infty}
p_{\eta_{n,i}}(r)dr+n^{-\frac{1+\alpha}{\alpha}}\int_{0}^{l_{n}^{\frac{1}{\alpha}}}r^{\alpha+1}\big\vert \frac{\alpha}{r^{\alpha+1}}-p_{\eta_{n,i}}(r)\big\vert dr+\mathcal{R}_{n,\alpha,i}\Big\},
\end{align*}
where
\begin{align*}
\mathcal{R}_{n,\alpha,i}=n^{-\frac{1}{\alpha}}\Big\vert \int_{\mathbb{S}^{d-1}}\theta\nu(d\theta)\Big\vert
\Big\vert \frac{\alpha}{1-\alpha}l_{n}^{\frac{1-\alpha}{\alpha}}-\int_{0}^{l_{n}^{\frac{1}{\alpha}}}rp_{\eta_{n,i}}(r)dr\Big\vert .
\end{align*}
\item[(4)] When $\alpha\in(0,\frac{1}{2}]$, there exists a finite constant $C>0$ depending on $\al,\beta,d,\nu$ such that
\begin{align*}
&\sup_{h\in\mathcal{H}_{\beta}\cap\mathcal{F}_{\beta}}\big\vert \mathbb{E}h(S_{n})-\mu(h)\big\vert \\
\leq&C\sum_{i=1}^{n}\Big\{n^{-\frac{3\alpha+1}{2\alpha}}\int_{0}^{l_{n}^{\frac{1}{\alpha}}}r^{\frac{\alpha+1}{2}}p_{\eta_{n,i}}(r)dr+n^{-\frac{\beta}{\alpha}}\int_{l_{n}^{\frac{1}{\alpha}}}^{\infty}t^{\beta}\big\vert d[\frac{\alpha}{t^{\alpha}}-tp_{\eta_{n,i}}(t)]\big\vert \\
&\quad+n^{-1}\int_{l_{n}^{\frac{1}{\alpha}}}^{\infty}
p_{\eta_{n,i}}(r)dr+n^{-\frac{1+\alpha}{\alpha}}\int_{0}^{l_{n}^{\frac{1}{\alpha}}}r^{\alpha+1}\big\vert \frac{\alpha}{r^{\alpha+1}}-p_{\eta_{n,i}}(r)\big\vert dr+\mathcal{R}_{n,\alpha,i}\Big\}.
\end{align*}
\end{itemize}
\end{theorem}

\begin{remark}
In order to use the Leave-one-out method to prove the above theorem, when $\alpha\in(0,\frac{1}{2}]$, we need better decaying property with respect to the solution of the Stein's equation (\ref{e:steinsequation}), hence we consider the case $h\in\mathcal{H}_{\beta}\cap\mathcal{F}_{\beta}$.
An example $\nu-$Paretian distribution for this theorem will be given in Section \ref{s:example}.
\end{remark}

\section{Proof of Theorem \ref{t:reg} and Proposition \ref{prop}}\label{s:reg}

\subsection{Proof of Theorem \ref{t:reg}}
Now we prove all the claims of Theorem \ref{t:reg}.

\begin{proof}[Proof of \eqref{first order}, \eqref{L9} and \eqref{firstb}]
For any $\al\in (0,2)$,  denote $s=(1-\e^{-t})$ and $z=y-\e^{-\frac{t}{\alpha}}x,$ it is easy to check
 $$
 \nabla_{x}p(s,z)=-\e^{-\frac{t}{\alpha}}\nabla_{z}p(s,z), \qquad \nabla_{y}p(s,z)=\nabla_{z}p(s,z).
 $$
Notice that $\|  \nabla h\|  _{\infty}\leq1,$ from which it is readily checked that one can differentiate under the integral sign in (\ref{solution}). Hence,
\begin{align}\label{e:nablaf}
 \nabla f(x)&=-\int_{0}^{\infty}\int_{\mathbb{R}^{d}}\nabla_{x}p(s,z)\big(h(y)-\mu(h)\big)\dif y\dif t\nonumber\\
 &=\int_{0}^{\infty}\int_{\mathbb{R}^{d}}\e^{-\frac{t}{\alpha}}\nabla_{z}p(s,z)\big(h(y)-\mu(h)\big)\dif y\dif t\\
 &=\int_{0}^{\infty}\int_{\mathbb{R}^{d}}\e^{-\frac{t}{\alpha}}\nabla_{y}p(s,z)\big(h(y)-\mu(h)\big)\dif y\dif t\nonumber\\
 &=-\int_{0}^{\infty}\int_{\mathbb{R}^{d}}\e^{-\frac{t}{\alpha}}p(s,z)\nabla h(y)\dif y\dif t\nonumber.
 \end{align}
Therefore,
 \begin{align*}
 \|  \nabla f\|  _\infty &\leq\|  \nabla h\|  _\infty \int_{0}^{\infty}\e^{-\frac{t}{\alpha}}\int_{\mathbb{R}^{d}}p(s,z)\dif y\dif t\\
 &=\|  \nabla h\|  _\infty \int_{0}^{\infty}\e^{-\frac{t}{\alpha}}\int_{\mathbb{R}^{d}}p(s,z)\dif z\dif t=\alpha\|  \nabla h\|  _\infty.
 \end{align*}
\end{proof}

\begin{proof}[\it Proof of \eqref{second order}]
From \eqref{e:nablaf} we see that
\begin{equation*}
 \|  \nabla^{2}f(x)\|  _\mathrm{op}\leq\int_{0}^{\infty}\int_{\mathbb{R}^{d}}\e^{-\frac{2t}{\alpha}}\vert \nabla_{z}p(s,z)\vert \vert \nabla h(y)\vert \dif y\dif t.
 \end{equation*}
 Thanks to the scaling property $p(s,z)=s^{-d/\alpha}p(s^{-1/\alpha}z)$, we have
 \begin{align*}
\|  \nabla^{2}f(x)\|  _\mathrm{op} &\leq\|  \nabla h\|  _\infty \int_{0}^{\infty}\e^{-\frac{2t}{\alpha}}\int_{\mathbb{R}^{d}}s^{-\frac{d+1}{\alpha}}\vert \nabla p(s^{-\frac{1}{\alpha}}z)\vert \dif y\dif t\\
 &=\|  \nabla h\|  _\infty\int_{0}^{\infty}s^{-1/\alpha}\e^{-\frac{2t}{\alpha}}\int_{\mathbb{R}^{d}}\vert \nabla p(u)\vert \dif u\dif t,
 \end{align*}
 where the equality is by taking $u=s^{-1/\alpha}z$. Then, by Lemma \ref{l:heatkernel} \eqref{fg4},  there exists a finite $C>0$ such that
  \begin{align*}
 \|  \nabla^{2}f(x)\|  _\mathrm{op}&\leq  C \|  \nabla h\|  _\infty  \int_{0}^{\infty}s^{-1/\alpha}\e^{-\frac{2t}{\alpha}}\dif t =C \|  \nabla h\|  _\infty  B\big(\frac{\alpha-1}{\alpha},\frac{2}{\alpha}\big),
 \end{align*}
 where $B(a,b)$ is the Beta function.
\end{proof}

\medskip
\begin{proof}[\it Proof of \eqref{e:fracder}]
Before proving \eqref{e:fracder}, we give another representation of the operator $\mathcal{L}^{\alpha,\nu}$.
Fix $\alpha\in(1,2).$ Let $f\in C^{2}(\mathbb{R}^{d})$ be such that $\|  \nabla^{2}f\|  _\infty +\|  \nabla f\|  _\infty<\infty.$ We have
\begin{align}\label{transform}
\mathcal{L}^{\alpha,\nu}f(x)=\frac{d_{\alpha}}{\alpha}\int_{\mathbb{S}^{d-1}}\int_{0}^{\infty}\frac{\langle\theta,\nabla f(x+u\theta)\rangle-\langle\theta,\nabla f(x)\rangle}{u^{\alpha}}du\nu(d\theta), \quad x\in\mathbb{R}^{d}.
\end{align}
Indeed, one can write
\begin{align*}
\mathcal{L}^{\alpha,\nu}f(x)&=d_{\alpha}\int_{\mathbb{S}^{d-1}}\int_{0}^{\infty}\int_{0}^{r}\frac{\langle\theta,\nabla f(x+u\theta)\rangle-\langle\theta,\nabla f(x)\rangle}{r^{1+\alpha}}dudr\nu(d\theta)\\
&=d_{\alpha}\int_{\mathbb{S}^{d-1}}\int_{0}^{\infty}\int_{u}^{\infty}\frac{\langle\theta,\nabla f(x+u\theta)\rangle-\langle\theta,\nabla f(x)\rangle}{r^{1+\alpha}}drdu\nu(d\theta)\\
&=\frac{d_{\alpha}}{\alpha}\int_{\mathbb{S}^{d-1}}\int_{0}^{\infty}\frac{\langle\theta,\nabla f(x+u\theta)\rangle-\langle\theta,\nabla f(x)\rangle}{u^{\alpha}}du\nu(d\theta),
\end{align*}
implying \eqref{transform}.  Using (\ref{transform}), we can write
\begin{align*}
&\frac{\alpha}{d_{\alpha}}\Big\vert \mathcal{L}^{\alpha,\nu}f(x)-\mathcal{L}^{\alpha,\nu}f(y)\Big\vert \\
=&\Big\vert\int_{\mathbb{S}^{d-1}}\int_{0}^{\infty}\frac{\langle\theta,\nabla f(x+u\theta)\rangle-\langle\theta,\nabla f(x)\rangle-\langle\theta,\nabla f(y+u\theta)\rangle+\langle\theta,\nabla f(y)\rangle}{u^{\alpha}}du\nu(d\theta)\Big\vert \\
\leq&\int_{\mathbb{S}^{d-1}}\int_{0}^{\infty}\frac{\left\vert \left\langle\theta,\nabla f(x+u\theta)-\nabla f(x)-\nabla f(y+u\theta)+\nabla f(y)\right\rangle\right\vert }{u^{\alpha}}du\nu(d\theta)\\
\leq&2\|  \nabla^2 f\|   \vert x-y\vert \int_{\mathbb{S}^{d-1}}\int_{\vert x-y\vert }^{\infty}\frac{1}{u^{\alpha}}du\nu(d\theta)
+2\|  \nabla^2 f\|   \int_{\mathbb{S}^{d-1}}\int_{0}^{\vert x-y\vert }\frac{1}{u^{\alpha-1}}du\nu(d\theta)\\
=&\frac{2\|  \nabla^2 f\|  _{\infty}}{(2-\alpha)(\alpha-1)}\vert x-y\vert^{2-\alpha},
\end{align*}
ending the proof.
\end{proof}

\begin{proof}[\it Proof of \eqref{L3} and \eqref{L1}]
Let $\al\in (0,1]$. Differentiating under the integral, we have
\begin{align}\label{e:f'}
\nabla f(x)=-\int_{0}^{\infty}\int_{\mathbb{R}^{d}}e^{-t/\alpha}p(u)\nabla h\big(s(t)^{1/\alpha}u+e^{-t/\alpha}x\big)dudt.
\end{align}
Choose $B=\vert x-y\vert ^{\alpha}.$ Applying successively (\ref{e:f'}), a change of variables, and Lemma \ref{l:heatkernel}, we get that
\begin{align*}
&\vert \nabla f(x)-\nabla f(y)\vert \\
\leq&\int_{0}^{\infty}\!\!\!\!e^{-t/\alpha}\!\!\!\int_{\mathbb{R}^{d}}\!\vert p\big(u-s(t)^{-1/\alpha}e^{-t/\alpha}x\big)-p\big(u-s(t)^{-1/\alpha}e^{-t/\alpha}y\big)\vert \vert \nabla h\big(us(t)^{1/\alpha}\big)\vert dudt\\
\leq& C \|  \nabla h\|  _{\infty}\int_{0}^{\infty}e^{-t/\alpha}\big((s(t)^{-1/\alpha}e^{-t/\alpha}\vert x-y\vert )\wedge1\big)dt\\
\leq& C \|  \nabla h\|  _{\infty}\int_{0}^{\infty}e^{-t/\alpha}\big(t^{-1/\alpha}\vert x-y\vert )\wedge1\big)dt\\
\leq&C \Big(\int_{0}^{B}e^{-t/\alpha}dt+\int_{B}^{\infty}e^{-t/\alpha}t^{-1/\alpha}dt\vert x-y\vert \Big)
\leq C \big(B+\int_{B}^{\infty}t^{-1/\alpha}e^{-t/\alpha}dt\vert x-y\vert \big),
\end{align*}
where in the third inequality, we use the fact that $s(t)^{-1/\alpha}e^{-t/\alpha}=(e^{t}-1)^{-1/\alpha}\leq t^{-1/\alpha}.$ If $\alpha=1,$ then for $B=\vert x-y\vert <1,$
\begin{align*}
\vert \nabla f(x)-\nabla f(y)\vert &\leq C \Big(B+\big(\int_{B}^{1}t^{-1}dt+\int_{1}^{\infty}e^{-t}dt\big)\vert x-y\vert \Big)\\
&\leq C(1-\log\vert x-z\vert )\vert x-z\vert .
\end{align*}
If $\alpha\in(0,1),$ then for $B=\vert x-y\vert ^{\alpha}$
\begin{align*}
\vert \nabla f(x)-\nabla f(y)\vert \leq C \big(B+\int_{B}^{\infty}t^{-1/\alpha}dt \vert x-z\vert \big)\leq C \vert x-y\vert ^{\alpha}.
\end{align*}
\end{proof}

\begin{proof}[\it Proof of \eqref{L4} and \eqref{L2}]
For $\alpha\in(0,1],$ one has by (\ref{solution})
\begin{multline*}
 f(x+w)-f(x) \\ =-\int_0^\infty \int_{\mathbb{R}^{d}}p(z)(h(s(t)^{-1/\alpha}z+e^{-t/\alpha}(x+w))-h(s(t)^{-1/\alpha}z+e^{-t/\alpha}x))\,dz\,dt.
\end{multline*}
Thus, for $h\in\mathcal{H}_{\beta}$ with $\beta\in(0,\alpha),$
\begin{equation*}
  \vert f(x+w)-f(x)\vert \le \int_0^\infty \int_{\mathbb{R}^{d}}p(z)e^{-\beta t/\alpha}\,dz\,dt(\vert w\vert ^\beta\wedge\vert w\vert )=\frac{\alpha}{\beta}(\vert w\vert ^\beta\wedge\vert w\vert ).
\end{equation*}
\end{proof}

\begin{proof}[\it Proof of \eqref{L7} and \eqref{L5}]
We first prove \eqref{L5}. By \eqref{L2}, for $\alpha \in (0,1)$,
\begin{align*}
\vert \mathcal{L}^{\alpha,\nu}f(x)\vert \le& d_{\alpha}\int_{\mathbb{S}^{d-1}}\int_{0}^{\infty}\frac{\vert f(x+r\theta)-f(x)\vert }{r^{\alpha+1}}dr\nu(d\theta)\\
\leq&\frac{\alpha d_{\alpha}}{\beta}\int_{\mathbb{S}^{d-1}}\int_{0}^{\infty}\frac{r^{\beta}\wedge r}{r^{\alpha+1}}dr\nu(d\theta)\leq C.
\end{align*}
Now it remains to bound $\|  \mathcal{L}^{1,\nu}f\|  _{\infty}.$ By \eqref{L3}, for $\vert w\vert \leq1,$ one has
\begin{align*}
\vert f(x+w)-f(x)-\langle\nabla f(x),w\rangle\vert &=\left\vert \int_{0}^{1}\left\langle\nabla f(x+uw)-\nabla f(x),w\right\rangle du\right\vert \\
&\le \vert w\vert \int_0^{1} \vert \nabla f(x+uw)-\nabla f(x)\vert  du \\
&\leq C\vert w\vert ^{2}\int_{0}^{1}u(1+\log\frac{1}{u\vert w\vert })du\\
&\leq C\vert w\vert ^2(1+\log\frac{1}{\vert w\vert }).
\end{align*}
It follows from (\ref{L4}) that
\begin{align*}
\vert \mathcal{L}^{1,\nu}f(x)\vert \leq C\Big(\int_{\mathbb{S}^{d-1}}\int_{0}^{1}(1+\log\frac{1}{r})dr\nu(d\theta)+\int_{\mathbb{S}^{d-1}}\int_{1}^{\infty}r^{\beta-2}dr\nu(d\theta)\Big)\leq C.
\end{align*}
\end{proof}

\begin{proof}[\it Proof of \eqref{decay1} and \eqref{decay}]
For $\alpha\in(0,1],$ one has  by  \eqref{e:steinsequation}  that
\begin{align*}
\vert \langle x,\nabla f(x)\vert =\alpha\big\vert \mathcal{L}^{\alpha,\nu}f(x)-[h(x)-h(0)]+[\pi(h)-h(0)]\big\vert .
\end{align*}
Thus, by (\ref{L7}) and (\ref{L5}), we have
\begin{align*}
\vert \langle x,\nabla f(x)\rangle\vert \leq C\big[1+\vert x\vert \wedge\vert x\vert ^{\beta}\big]\leq C(1+\vert x\vert ^{\beta}).
\end{align*}
The proof is complete.
\end{proof}

It remains to prove \eqref{L8} and \eqref{L6}. We need two lemmas.

\bl\label{lem3}
Let $\alpha\in(0,1]$ and $h\in\mathcal{H}_{\beta}$ with $\beta\in(0,\alpha).$
\begin{itemize}
\item[1)] If $\alpha=1$, then for any $a>0$,
$$\left\vert \int_{\mathbb{R}^{d}}\mathcal{L}^{1,\nu}p(y)h(ay)\,dy\right\vert \le C(a^{\beta}+a).$$
\item[2)] If $\alpha\in(0,1),$ then for any $a>0,$
$$\left\vert \int_{\mathbb{R}^{d}}\mathcal{L}^{\alpha,\nu}p(y)h(ay)\,dy\right\vert \le Ca^{\alpha}.$$
\end{itemize}
\el
\begin{proof}
1) Notice that for any $y\in\mathbb{R}^{d}$,
\begin{align*}
\mathcal{L}^{1,\nu}p(y)=\int_{\mathbb{S}^{d-1}}\int_{1}^{\infty}&\frac{p(y+r\theta)-p(y)}{r^{2}}dr\nu(d\theta)\\
&+\int_{\mathbb{S}^{d-1}}\int_{0}^{1}\frac{p(y+r\theta)-p(y)-\langle r\theta,\nabla p(y)\rangle}{r^{2}}dr\nu(d\theta).
\end{align*}
By Fubuni's theorem,
\begin{align*}
&\Big\vert \int_{\mathbb{R}^{d}}\int_{\mathbb{S}^{d-1}}\int_{1}^{\infty}\frac{p(y+r\theta)-p(y)}{r^{2}}dr\nu(d\theta)h(ay)dy\Big\vert \\
=&\Big\vert \int_{\mathbb{R}^{d}}\int_{\mathbb{S}^{d-1}}\int_{1}^{\infty}\frac{p(y)}{r^{2}}\big(h(ay-ar\theta)-h(ay)\big)dr\nu(d\theta)dy\Big\vert \\
\leq&a^{\beta}\int_{\mathbb{S}^{d-1}}\int_{1}^{\infty}\int_{\mathbb{R}^{d}}\frac{p(y)r^{\beta}}{r^{2}}dydr\nu(d\theta)\leq Ca^{\beta}.
\end{align*}
Applying Fubini's theorem, integration by parts and the estimate of $\nabla p(x)$ (Lemma \ref{l:heatkernel}), we get
\begin{align*}
&\Big\vert \int_{\mathbb{R}^{d}}\int_{\mathbb{S}^{d-1}}\int_{0}^{1}\frac{p(y+r\theta)-p(y)-\langle r\theta,\nabla p(y)\rangle}{r^{2}}dr\nu(d\theta)h(ay)dy\Big\vert \\
=&\Big\vert \int_{\mathbb{R}^{d}}\int_{\mathbb{S}^{d-1}}\int_{0}^{1}\int_{0}^{1}\frac{\langle\nabla p(y+ur\theta)-\nabla p(y), r\theta\rangle}{r^{2}}dudr\nu(d\theta)h(ay)dy\Big\vert \\
\leq&a\|  \nabla h\|  _{\infty}\Big\vert \int_{\mathbb{R}^{d}}\int_{\mathbb{S}^{d-1}}\int_{0}^{1}\int_{0}^{1}\frac{\vert p(y+ur\theta)-p(y)\vert }{r}dudr\nu(d\theta)dy\Big\vert \\
\leq&Ca\|  \nabla h\|  _{\infty}\int_{\mathbb{S}^{d-1}}\int_{0}^{1}\int_{0}^{1}\frac{ur}{r}dudr\nu(d\theta)\leq Ca.
\end{align*}
2) We have by Fubini's theorem that, for any $a>0,$
\begin{align*}
&\left\vert \int_{\mathbb{R}^{d}}(\mathcal{L}^{\alpha,\nu}p)(y)h(ay)\,dy\right\vert\\ 
=&d_{\alpha}\Big\vert \int_{\mathbb{R}^{d}}\int_{\mathbb{S}^{d-1}}
\int_{0}^{\infty}\frac{p(y+r\theta)-p(y)}{r^{\alpha+1}}dr\nu(d\theta)h(ay)dy\Big\vert \\
=&d_{\alpha}\Big\vert \int_{\mathbb{R}^{d}}\int_{\mathbb{S}^{d-1}}
\int_{0}^{\infty}\frac{p(y)}{r^{\alpha+1}}\big(h(ay-ar\theta)-h(ay)\big)dr\nu(d\theta)dy\Big\vert \\
=&d_{\alpha}a^{\alpha}\Big\vert \int_{\mathbb{R}^{d}}\int_{\mathbb{S}^{d-1}}
\int_{0}^{\infty}\frac{p(y)}{u^{\alpha+1}}\big(h(ay-u\theta)-h(ay)\big)du\nu(d\theta)dy\Big\vert \\
\leq& d_{\alpha}a^{\alpha}\int_{\mathbb{S}^{d-1}}
\int_{0}^{\infty}\int_{\mathbb{R}^{d}}\frac{p(y)(u\wedge u^{\beta})}{u^{\alpha+1}}dydu\nu(d\theta)\leq Ca^{\alpha}.
\end{align*}
Thus, the assertion is proved.
\end{proof}

\bl\label{lem:tria-p}
Let $\alpha\in(0,1].$ Then
\begin{align*}
\int_{\mathbb{R}^{d}}\vert \mathcal{L}^{\alpha,\nu}p(z)\vert dz\leq C.
\end{align*}
\el
\begin{proof}
If $\alpha=1,$ then by Lemma \ref{l:heatkernel}, we get that, for any $\vert u\vert \leq1,$
\begin{align*}
\vert \nabla^{2}p(z+u)\vert \leq\frac{C_{d}}{(1+\vert z+u\vert )^{\gamma+1}}\leq\frac{C}{(1+\vert z\vert )^{\gamma+1}},
\end{align*}
where in the last inequality, we use the fact that $2(1+\vert z+u\vert )\geq2+\vert z\vert -\vert u\vert \geq1+\vert z\vert .$ It follows that, for any $\vert w\vert \leq1,$
\begin{align*}
\vert p(z+w)-p(z)-w\cdot\nabla p(z)\vert \leq\frac{C}{(1+\vert z\vert )^{\gamma+1}}\vert w\vert ^{2}.
\end{align*}
Thus, we have that
\begin{align*}
\int_{\mathbb{R}^{d}}\vert \mathcal{L}^{1,\nu}p(z)\vert dz&\leq d_{1}\int_{\mathbb{R}^{d}}\int_{\mathbb{S}^{d-1}}\int_{1}^{\infty}\frac{p(z+r\theta)+p(z)}{r^{2}}dr\nu(d\theta)dz\\
&\quad+d_{1}\int_{\mathbb{R}^{d}}\int_{\mathbb{S}^{d-1}}\int_{0}^{1}\frac{\vert p(z+r\theta)-p(z)-\langle r\theta,\nabla p(z)\rangle\vert }{r^{2}}dr\nu(d\theta)dz\\
&\leq2d_{1}+d_{1}\int_{\mathbb{R}^{d}}\int_{\mathbb{S}^{d-1}}\int_{0}^{1}\frac{C}{(1+\vert z\vert )^{\gamma+1}}dr\nu(d\theta)dz\leq C.
\end{align*}
If $\alpha\in(0,1),$ then by Lemma \ref{l:heatkernel},
\begin{align*}
\int_{\mathbb{R}^{d}}\vert \mathcal{L}^{\alpha,\nu}p(z)\vert dz&\leq d_{\alpha}\int_{\mathbb{R}^{d}}\int_{\mathbb{S}^{d-1}}\int_{0}^{\infty}\frac{\vert p(z+r\theta)-p(z)\vert }{r^{\alpha+1}}dr\nu(d\theta)dz\\
&\leq C\int_{\mathbb{S}^{d-1}}\int_{0}^{\infty}\frac{r\wedge1}{r^{\alpha+1}}dr\nu(d\theta)\leq C.
\end{align*}
\end{proof}

We are ready to complete the proof of Theorem \ref{t:reg}.

\begin{proof}[\it Proof of \eqref{L8} and \eqref{L6}]
Set $s(t)=1-e^{-t}$ and $\tilde{h}=h-\mathbb{E}[h(Z)].$ We claim that
\begin{eqnarray}\label{1.2}
  &&\mathcal{L}^{\alpha,\nu}f(x)=-\int_0^\infty \int_{\mathbb{R}^{d}} \mathcal{L}^{\alpha,\nu}q(t,\cdot,y)(x)\widetilde h(y)\,dy\,dt\nonumber\\
  &=&-\int_0^\infty s(t)^{-1}e^{-t}\,dt\int_{\mathbb{R}^{d}}(\mathcal{L}^{\alpha,\nu}p)(z)\widetilde h\left(s(t)^{1/\alpha}z+e^{-t/\alpha}x\right)\,dz.
\end{eqnarray}
The second equality follows from (\ref{density}). To see that the first one holds, note that $h\in\mathcal{H}_{\beta}$ so that Fubini's theorem implies
\begin{align*}
\mathcal{L}^{\alpha,\nu} f(x) = -\int_0^\infty \mathcal L^{\alpha,\nu} \left(\int_{\mathbb{R}^{d}} q(t,\cdot,y) \widetilde h(y) dy\right) (x) dt.
\end{align*}
For each fixed $t>0,$ applying Lemma \ref{l:heatkernel} justifies a further use of Fubini's theorem, we are led to
\begin{align*}
\mathcal L^{\alpha,\nu} \left(\int_{\mathbb{R}^{d}} q(t,\cdot,y) \widetilde h(y) dy\right) (x) = \int_{\mathbb{R}^{d}} \mathcal L^{\alpha,\nu} q(t,\cdot,y)(x)\widetilde h(y) dy,
\end{align*}
and the claim follows. Now let $\alpha=1$ and let $x,y\in\RR^d$ be such that $\vert x-y\vert \le 1$. By Lemma \ref{lem:tria-p}, we get that
\begin{align*}
  &\left\vert \int_{\mathbb{R}^{d}}(\mathcal{L}^{1,\nu}p)(z)(h(s(t)^{1/\alpha}z+e^{-t/\alpha}x)-h(s(t)^{1/\alpha}z+e^{-t/\alpha}y))\,dz\right\vert \\
&\le e^{-t}\vert x-y\vert \int_{\mathbb{R}^{d}}\left\vert (\mathcal{L}^{1,\nu}p)(z)\right\vert \,dz.
\end{align*}
By Lemma \ref{lem3}, we get that, for $t<1$,
\begin{align}\label{eq:p2}
 &\left\vert \int_{\mathbb{R}^{d}}(\mathcal{L}^{1,\nu}p)(z)h\left(s(t)^{1/\alpha}z+e^{-t/\alpha}x\right)\,dz\!-\!\int_{\mathbb{R}^{d}}(\mathcal{L}^{1,\nu}p)(z)h\left(s(t)^{1/\alpha}z+e^{-t/\alpha}y\right)\,dz\right\vert \nonumber\\
&\le C(s(t)+s(t)^{\beta})\le Cs(t)^{\beta}.
\end{align}
Let $B=\vert x-y\vert ^{1/\beta}$. Then by \eqref{1.2}, we get that
\begin{align*}
\vert \mathcal{L}^{1,\nu}f(x)-\mathcal{L}^{1,\nu}f(y)\vert &\le C\Big(\int_0^B s(t)^{-1}s(t)^{\beta}e^{-t}\,dt+\int_B^\infty s(t)^{-1}e^{-2t}\,dt \vert x-y\vert \Big)\\
&\le C\Big(B^{\beta}+\int_B^\infty t^{-1}e^{-t}\,dt \vert x-y\vert \Big)\\
&=C\Big(B^{\beta}+\left(\int_B^1 t^{-1}\,dt+\int_1^\infty e^{-t}\,dt\right) \vert x-y\vert \Big)\\
&\leq C\vert x-y\vert \left(1-\log\vert x-y\vert \right),
\end{align*}
If $\al\in (0,1)$, by Lemma \ref{lem:tria-p}, we get that
\begin{align*}
  &\quad\left\vert \int_{\mathbb{R}^{d}}(\mathcal{L}^{\alpha,\nu}p)(z)(\widetilde h(s(t)^{1/\alpha}z+e^{-t/\alpha}x)-\widetilde h(s(t)^{1/\alpha}z+e^{-t/\alpha}y))\,dz\right\vert \\
&\le e^{-t/\alpha}\vert x-y\vert \int_{\mathbb{R}^{d}}\left\vert (\mathcal{L}^{\alpha,\nu}p)(z)\right\vert \,dz
   \le C\vert x-y\vert .
\end{align*}
By Lemma \ref{lem3}  applied to $\widetilde h(\cdot+e^{-t/\alpha}x), \widetilde h(\cdot+e^{-t/\alpha}y)\in \mathcal H_\beta$, we get that,
\begin{align*}
&\left\vert \int_{\mathbb{R}^{d}}(\mathcal{L}^{\alpha,\nu}p)(z)\widetilde h\left((s(t)^{1/\alpha}z+e^{-t/\alpha}x\right)\!dz\!-\!\!\int_{\mathbb{R}^{d}}(\mathcal{L}^{\alpha,\nu}p)(z)\widetilde h\left((s(t)^{1/\alpha}z+e^{-t/\alpha}y\right)\!dz\right\vert\\ 
&\le Cs(t).
\end{align*}
Thus, we get that, for any $\eta\in[0,1]$,
\begin{align}\label{1.1}
  &\left\vert \int_{\mathbb{R}^{d}}(\mathcal{L}^{\alpha,\nu}p)(z)\widetilde h\left(s(t)^{1/\alpha}z+e^{-t/\alpha}x\right)\!dz\!-\!\int_{\mathbb{R}^{d}}(\mathcal{L}^{\alpha,\nu}p)(z)\widetilde h\left(s(t)^{1/\alpha}z+e^{-t/\alpha}y\right)\!dz\right\vert \nonumber\\
  &\le C\big(s(t)\wedge\vert x-y\vert \big)\le Cs(t)(1\wedge s(t)^{-1}\vert x-y\vert )^{\eta}\leq Cs(t)^{1-\eta}\vert x-y\vert ^{\eta}.
\end{align}
Then, by \eqref{1.2} and \eqref{1.1}, we get that
\begin{align*}
 \quad\vert \mathcal{L}^{\alpha,\nu}f(x)-\mathcal{L}^{\alpha,\nu}f(y)\vert &\le C\int_0^\infty s(t)^{-1}s(t)^{1-\eta}e^{-t}\,dt \vert x-y\vert ^{\eta}\\
  &\le C\int_0^\infty s(t)^{-\eta}e^{-t}\,dt \vert x-y\vert ^{\eta}\\
  &\le C\int_0^\infty t^{-\eta}e^{-(1-\eta)t}\,dt \vert x-y\vert ^{\eta}\leq C\vert x-y\vert ^{\eta},
\end{align*}
completing the proof.
\end{proof}

\subsection{Proof of Proposition \ref{prop}}
When $\vert x\vert <1$, the conclusion is obvious.

When $\vert x\vert \geq1$, notice that
\begin{align*}
\nabla f(x)=-\int_{0}^{\infty}\int_{\mathbb{R}^{d}}e^{-t/\alpha}p(y)\nabla h\big(s(t)^{1/\alpha}y+e^{-t/\alpha}x\big)dydt,
\end{align*}
as $\vert y\vert \geq2s(t)^{-1/\alpha}e^{-t/\alpha}\vert x\vert $, that is, $s(t)^{1/\alpha}\vert y\vert \geq2e^{-t/\alpha}\vert x\vert $, we have
\begin{align*}
&e^{\frac{\beta-1}{\alpha}t}\vert x\vert ^{1-\beta}\vert \nabla h\big(s(t)^{1/\alpha}y+e^{-t/\alpha}x\big)\vert \\
\leq&\frac{e^{\frac{\beta-1}{\alpha}t}\vert x\vert ^{1-\beta}}{1+\left\vert s(t)^{1/\alpha}y+e^{-t/\alpha}x\right\vert ^{1-\beta}}\\
\leq&\frac{e^{\frac{\beta-1}{\alpha}t}\vert x\vert ^{1-\beta}}{1+\left(s(t)^{1/\alpha}\vert y\vert -e^{-t/\alpha}\vert x\vert \right)^{1-\beta}}\leq\frac{e^{\frac{\beta-1}{\alpha}t}\vert x\vert ^{1-\beta}}{1+e^{\frac{\beta-1}{\alpha}t}\vert x\vert ^{1-\beta}}\leq1,
\end{align*}
which implies
\begin{align*}
\vert \nabla h\big(s(t)^{1/\alpha}y+e^{-t/\alpha}x\big)\vert \leq e^{\frac{1-\beta}{\alpha}t}\vert x\vert ^{\beta-1}.
\end{align*}
These imply
\begin{align*}
&\left\vert \int_{0}^{\infty}\int_{\vert y\vert \geq2s(t)^{-1/\alpha}e^{-t/\alpha}\vert x\vert }e^{-\frac{t}{\alpha}}p(y)\nabla h\big(s(t)^{1/\alpha}y+e^{-t/\alpha}x\big)dydt\right\vert \\
\leq&\int_{0}^{\infty}\int_{\vert y\vert \geq2s(t)^{-1/\alpha}e^{-t/\alpha}\vert x\vert }e^{-\frac{t}{\alpha}}p(y)e^{\frac{1-\beta}{\alpha}t}\vert x\vert ^{\beta-1}dydt\\
\leq&\vert x\vert ^{\beta-1}\int_{0}^{\infty}\int_{\mathbb{R}^{d}}e^{-\frac{\beta}{\alpha}t}p(y)dydt=\frac{\alpha}{\beta}\vert x\vert ^{\beta-1}.
\end{align*}
As $\vert y\vert \leq\frac{1}{2}s(t)^{-1/\alpha}e^{-t/\alpha}\vert x\vert $, that is, $s(t)^{1/\alpha}\vert y\vert \leq\frac{1}{2}e^{-t/\alpha}\vert x\vert $, we have
\begin{align*}
e^{\frac{\beta-1}{\alpha}t}\vert x\vert ^{1-\beta}\vert \nabla h\big(s(t)^{1/\alpha}y+e^{-t/\alpha}x\big)\vert \leq&\frac{e^{\frac{\beta-1}{\alpha}t}\vert x\vert ^{1-\beta}}{1+\left(e^{-t/\alpha}\vert x\vert -s(t)^{1/\alpha}\vert y\vert \right)^{1-\beta}}\\
\leq&\frac{e^{\frac{\beta-1}{\alpha}t}\vert x\vert ^{1-\beta}}{1+2^{\beta-1}e^{\frac{\beta-1}{\alpha}t}\vert x\vert ^{1-\beta}}\leq2,
\end{align*}
which implies
\begin{align*}
\vert \nabla h\big(s(t)^{1/\alpha}y+e^{-t/\alpha}x\big)\vert \leq2e^{\frac{1-\beta}{\alpha}t}\vert x\vert ^{\beta-1}.
\end{align*}
These imply
\begin{align*}
\left\vert \int_{0}^{\infty}\int_{\vert y\vert \leq\frac{1}{2}s(t)^{-1/\alpha}e^{-t/\alpha}\vert x\vert }e^{-\frac{t}{\alpha}}p(y)\nabla h\big(s(t)^{1/\alpha}y+e^{-t/\alpha}x\big)dydt\right\vert
\leq\frac{2\alpha}{\beta}\vert x\vert ^{\beta-1}.
\end{align*}
As $\frac{1}{2}s(t)^{-1/\alpha}e^{-t/\alpha}\vert x\vert \leq\vert y\vert \leq2s(t)^{-1/\alpha}e^{-t/\alpha}\vert x\vert $, (\ref{est:p}) implies
\begin{align*}
&\left\vert \int_{0}^{\infty}\!\!\!\int_{\frac{1}{2}s(t)^{-1/\alpha}e^{-t/\alpha}\vert x\vert \leq\vert y\vert \leq2s(t)^{-1/\alpha}e^{-t/\alpha}\vert x\vert }\!\!e^{-\frac{t}{\alpha}}p(y)\nabla h\big(s(t)^{1/\alpha}y+e^{-t/\alpha}x\big)dydt\right\vert \\
&\leq C\left[\int_{0}^{\ln(1+\vert x\vert ^{\alpha})}\!\!\!e^{-\frac{t}{\alpha}}\!\int_{\frac{1}{2}e^{-t/\alpha}\vert x\vert \leq s(t)^{1/\alpha}\vert y\vert \leq2e^{-t/\alpha}\vert x\vert }\frac{\vert y\vert ^{-(\alpha+\gamma)}}{\left\vert s(t)^{1/\alpha}y+e^{-t/\alpha}x\right\vert ^{1-\beta}}dydt\right.\\
&\left.\qquad\qquad+\int_{\ln\left(1+\vert x\vert ^{\alpha}\right)}^{\infty}e^{-\frac{t}{\alpha}}\int_{\frac{1}{2}s(t)^{-1/\alpha}e^{-t/\alpha}\vert x\vert \leq\vert y\vert \leq2s(t)^{-1/\alpha}e^{-t/\alpha}\vert x\vert }dydt\right].
\end{align*}
For the first term, we have
\begin{align*}
&\int_{0}^{\ln\left(1+\vert x\vert ^{\alpha}\right)}e^{-\frac{t}{\alpha}}\int_{\frac{1}{2}e^{-t/\alpha}\vert x\vert \leq s(t)^{1/\alpha}\vert y\vert \leq2e^{-t/\alpha}\vert x\vert }\frac{\vert y\vert ^{-(\alpha+\gamma)}}{\left\vert s(t)^{1/\alpha}y+e^{-t/\alpha}x\right\vert ^{1-\beta}}dydt\\
\leq&\int_{0}^{\ln\left(1+\vert x\vert ^{\alpha}\right)}e^{-\frac{t}{\alpha}}\int_{\mathbb{S}^{d-1}}\int_{\frac{1}{2}s(t)^{-1/\alpha}e^{-t/\alpha}\vert x\vert }^{2s(t)^{-1/\alpha}e^{-t/\alpha}\vert x\vert }\frac{r^{-(\alpha+1+\gamma-d)}}{\left\vert s(t)^{1/\alpha}r\theta-e^{-t/\alpha}\vert x\vert \right\vert ^{1-\beta}}drd\theta dt\\
\leq&C\int_{0}^{\ln\left(1+\vert x\vert ^{\alpha}\right)}e^{t}s(t)^{\frac{\alpha+1}{\alpha}}\int_{\frac{1}{2}s(t)^{-1/\alpha}e^{-t/\alpha}\vert x\vert }^{s(t)^{-1/\alpha}e^{-t/\alpha}\vert x\vert }\frac{\vert x\vert ^{-(\alpha+\gamma+1-d)}}{\left(e^{-t/\alpha}\vert x\vert -s(t)^{1/\alpha}r\right)^{1-\beta}}drdt\\
&+C\int_{0}^{\ln\left(1+\vert x\vert ^{\alpha}\right)}e^{t}s(t)^{\frac{\alpha+1}{\alpha}}\int_{s(t)^{-1/\alpha}e^{-t/\alpha}\vert x\vert }^{2s(t)^{-1/\alpha}e^{-t/\alpha}\vert x\vert }\frac{\vert x\vert ^{-(\alpha+\gamma+1-d)}}{\left(s(t)^{1/\alpha}r-e^{-t/\alpha}\vert x\vert \right)^{1-\beta}}drdt.
\end{align*}
Then, one can write
\begin{align*}
&\int_{0}^{\ln\left(1+\vert x\vert ^{\alpha}\right)}e^{t}s(t)^{\frac{\alpha+1}{\alpha}}\int_{\frac{1}{2}s(t)^{-1/\alpha}e^{-t/\alpha}\vert x\vert }^{s(t)^{-1/\alpha}e^{-t/\alpha}\vert x\vert }\frac{\vert x\vert ^{-(\alpha+\gamma+1-d)}}{\left(e^{-t/\alpha}\vert x\vert -s(t)^{1/\alpha}r\right)^{1-\beta}}drdt\\
=&\frac{1}{\vert x\vert ^{\alpha+\gamma+1-d}}\int_{0}^{\ln\left(1+\vert x\vert ^{\alpha}\right)}e^{t}s(t)\int^{\frac{1}{2}e^{-t/\alpha}\vert x\vert }_{0}\frac{1}{r^{1-\beta}}drdt\\
\leq&\frac{C}{\vert x\vert ^{\alpha+\gamma+1-d-\beta}}\int_{0}^{\ln\left(1+\vert x\vert ^{\alpha}\right)}e^{\frac{\alpha-\beta}{\alpha}t}\dif t\\
\leq&\frac{C}{\vert x\vert ^{\alpha+\gamma+1-d-\beta}}e^{\frac{\alpha-\beta}{\alpha}\ln(1+\vert x\vert ^{\alpha})}\leq\frac{C}{\vert x\vert ^{\gamma+1-d}},
\end{align*}
whereas
\begin{align*}
&\int_{0}^{\ln\left(1+\vert x\vert ^{\alpha}\right)}e^{t}s(t)^{\frac{\alpha+1}{\alpha}}\int_{s(t)^{-1/\alpha}e^{-t/\alpha}\vert x\vert }^{2s(t)^{-1/\alpha}e^{-t/\alpha}\vert x\vert }\frac{\vert x\vert ^{-(\alpha+\gamma+1-d)}}{\left(s(t)^{1/\alpha}r-e^{-t/\alpha}\vert x\vert \right)^{1-\beta}}drdt\\
\leq&\frac{C}{\vert x\vert ^{\gamma+1-d}}.
\end{align*}
These imply
\begin{align*}
&\int_{0}^{\ln\left(1+\vert x\vert ^{\alpha}\right)}e^{-\frac{t}{\alpha}}\int_{\frac{1}{2}e^{-t/\alpha}\vert x\vert \leq s(t)^{1/\alpha}\vert y\vert \leq2e^{-t/\alpha}\vert x\vert }\frac{\vert y\vert ^{-(\alpha+\gamma)}}{\left\vert s(t)^{1/\alpha}y+e^{-t/\alpha}x\right\vert ^{1-\beta}}dydt\\
\leq&\frac{C}{\vert x\vert ^{\gamma+1-d}}.
\end{align*}
For the second term,
\begin{align*}
&\int_{\ln\left(1+\vert x\vert ^{\alpha}\right)}^{\infty}e^{-\frac{t}{\alpha}}\int_{\frac{1}{2}s(t)^{-1/\alpha}e^{-t/\alpha}\vert x\vert \leq\vert y\vert \leq2s(t)^{-1/\alpha}e^{-t/\alpha}\vert x\vert }dydt\\
\leq&C\vert x\vert \int_{\ln\left(1+\vert x\vert ^{\alpha}\right)}^{\infty}e^{-\frac{t}{\alpha}}(e^{t}-1)^{-1/\alpha}\dif t\\
\leq&\frac{C\vert x\vert }{\vert x\vert }\int_{\ln\left(1+\vert x\vert ^{\alpha}\right)}^{\infty}e^{-\frac{t}{\alpha}}\dif t\leq Ce^{-\frac{1}{\alpha}\ln\left(1+\vert x\vert ^{\alpha}\right)}\leq C\vert x\vert ^{-1}.
\end{align*}
The proof is complete.

\section{Proof of Theorem \ref{t:bound}}\label{s:bound}
\subsection{Alternate expressions for $\mathcal L^{\alpha,\nu}$}

The following lemma gathers useful alternate expressions for the operator $\mathcal{L}^{\alpha,\nu}$.

\bl\label{properties}
Let $\alpha\in(0,2)$ and
 $f\in C^2(\mathbb{R}^{d})$.
We have, for all $x\in\mathbb{R}^{d}$ and $a>0$,

\noindent
a.) When $\alpha\in(1,2)$,
\begin{align*}
(\mathcal{L}^{\alpha,\nu}f)(x)=&\frac{d_\al}{\alpha}\int_{\mathbb S^{d-1}} \int_0^\infty\frac{\langle\theta,\nabla f(x+ u\theta)\rangle-\langle\theta,\nabla f(x)\rangle}{u^{\alpha}}du  \nu(d\theta)\\
=&\frac{d_\al a^{1-\alpha}}{\alpha}\int_{\mathbb S^{d-1}} \int_0^\infty\frac{\langle u\theta,\nabla f(x+au\theta)\rangle-\langle u\theta,\nabla f(x)\rangle}{u^{\alpha+1}}du  \nu(d\theta)
\end{align*}
provided that $\|  \nabla f\|  _{\infty}<\infty$ and $\sup_{x\in\mathbb{R}^{d}}\|  \nabla^{2}f(x)\|  _{{\rm op}}<\infty$.

\noindent
b.) When $\alpha=1$,
\begin{align*}
\quad(\mathcal{L}^{1,\nu}f)(x)=&d_{1}\int_{\mathbb{S}^{d-1}}\int_{0}^{\infty}\frac{\langle\theta,\nabla f(x+u\theta)\rangle-\langle\theta,\nabla f(x){\bf 1}_{(0,1]}(u)\rangle}{u}du\nu(d\theta)\\
=&d_{1}\int_{\mathbb{S}^{d-1}}\int_{0}^{\infty}\frac{\langle u\theta,\nabla f(x+au\theta)\rangle-\langle u\theta,\nabla f(x){\bf 1}_{(0,1]}(u)\rangle}{u^{2}}du\nu(d\theta)
\end{align*}
provided that $\int_{\mathbb{S}^{d-1}}\int_{0}^{\infty}\frac{\vert f(x+u\theta)-f(x)-\langle u\theta,\nabla f(x){\bf 1}_{(0,1]}(u)\rangle\vert }{u^{2}}du\nu(d\theta)<\infty$ and\\ $\int_{\mathbb{S}^{d-1}}\int_{0}^{\infty}\frac{\vert \langle\theta,\nabla f(x+u\theta)\rangle-\langle\theta,\nabla f(x){\bf 1}_{(0,1]}(u)\rangle\vert }{u}du\nu(d\theta)<\infty.$

\noindent
c.)  When $\alpha\in(0,1),$
\begin{eqnarray*}
(\mathcal{L}^{\alpha,\nu}f)(x)&=&\frac{d_{\alpha}}{\alpha}\int_{\mathbb{S}^{d-1}}\int_{0}^{\infty}\frac{\langle\theta,\nabla f(x+u\theta)\rangle}{u^{\alpha}}du\nu(d\theta)\\
&=&\frac{d_{\alpha}a^{1-\alpha}}{\alpha}\int_{\mathbb{S}^{d-1}}\int_{0}^{\infty}\frac{\langle u\theta,\nabla f(x+au\theta)\rangle}{u^{\alpha+1}}du\nu(d\theta).
\end{eqnarray*}
provided $\int_{\mathbb{S}^{d-1}}\int_{0}^{\infty}\frac{\vert f(x+u\theta)-f(x)\vert }{u^{\alpha+1}}du\nu(d\theta)<\infty$, $\int_{\mathbb{S}^{d-1}}\int_{0}^{\infty}\frac{\vert \langle\theta,\nabla f(x+u\theta)\rangle\vert }{u^{\alpha}}du\nu(d\theta)<\infty$.
\el
\begin{proof}
Note that the conditions on $f$ ensure that all the integrals are well defined and we can use Fubini's theorem in the following proof.
\begin{align*}
(\mathcal{L}^{\alpha,\nu}f)(x)=&d_\al    \int_{\mathbb S^{d-1}} \int_0^\infty \left( f(x+ r\theta) - f(x) - k_\al(r) \langle r\theta, \nabla f(x) \rangle  \right)  \frac{dr}{r^{1+\al}}\nu(d\theta)\\
=&d_\al    \int_{\mathbb S^{d-1}} \int_0^\infty\int_{0}^{r}\left(\langle\theta,\nabla f(x+ u\theta)\rangle-k_\al(r)\langle\theta,\nabla f(x)\rangle\right)du  \frac{dr}{r^{1+\al}}\nu(d\theta)\\
=&d_\al    \int_{\mathbb S^{d-1}} \int_0^\infty\int_{u}^{\infty}\left(\langle\theta,\nabla f(x+ u\theta)\rangle-k_\al(r)\langle\theta,\nabla f(x)\rangle\right)\frac{dr}{r^{1+\al}}du  \nu(d\theta),
\end{align*}
since $k_\al(r)=\mathbf{1}_{\al=1, r\in (0,1)} + \mathbf{1}_{\al\in(1,2)}$ and $\int_{\mathbb{S}^{d-1}}\theta\nu(d\theta)=0$ when $\a=1$, we further have
\begin{align*}
(\mathcal{L}^{\alpha,\nu}f)(x)=&\frac{d_\al}{\alpha}\int_{\mathbb S^{d-1}} \int_0^\infty\frac{\langle\theta,\nabla f(x+ u\theta)-k_\al(u)\langle\theta,\nabla f(x)\rangle}{u^{\alpha}}du  \nu(d\theta)\\
=&\frac{d_\al a^{1-\alpha}}{\alpha}\int_{\mathbb S^{d-1}} \int_0^\infty\frac{\langle u\theta,\nabla f(x+au\theta)-k_\al(u)\langle u\theta,\nabla f(x)\rangle}{u^{\alpha+1}}du  \nu(d\theta).
\end{align*}
\end{proof}

Now we check that the solution $f$ to Stein's equation satisfies the integrability condition of the previous proposition.

\bl\label{alter}
Let $\alpha \in(\frac{1}{2},1]$ and $ h\in \mathcal H_\beta$ with $\beta \in (0,\alpha)$. Let $f$ be
defined as \eqref{solution}.
If $\alpha=1$ and $\int_{\mathbb{S}^{d-1}}\theta\nu(d\theta)=0,$ then for any $x\in\mathbb{R}^{d},$
\begin{align*}
\int_{\mathbb{S}^{d-1}}\int_{0}^{\infty}\frac{\vert \langle\theta,\nabla f(x+r\theta)\rangle-\langle\theta,\nabla f(x){\bf 1}_{(0,1]}(r)\rangle\vert }{r}\dif r\dif\theta\leq C(1+\ln\vert x\vert ).
\end{align*}

If $\alpha\in(0,1)$, then for any $x\in\mathbb{R}^{d},$
\begin{align*}
\int_{\mathbb{S}^{d-1}}\int_{0}^{\infty}\frac{\vert \langle\theta,\nabla f(x+r\theta)\rangle\vert }{r^\alpha}\dif r\dif\theta\leq C(1+\vert x\vert ^{1-\alpha}).
\end{align*}
\el

\begin{proof}
When $\alpha=1$ and $\vert x\vert <1$, we have
\begin{align*}
&\int_{\mathbb{S}^{d-1}}\int_{0}^{\infty}\frac{\vert \langle\theta,\nabla f(x+r\theta)\rangle-\langle\theta,\nabla f(x){\bf 1}_{(0,1]}(r)\rangle\vert }{r}\dif r\dif\theta\\
=& \int_{\mathbb{S}^{d-1}}\int_{0}^{1}\frac{\vert \langle\theta,\nabla f(x+r\theta)\rangle-\langle\theta,\nabla f(x)\rangle\vert }{r}\dif r\dif\theta+\int_{\mathbb{S}^{d-1}}\int_{1}^{\infty}\frac{\vert \langle\theta,\nabla f(x+r\theta)\rangle\vert }{r}\dif r\dif\theta.
\end{align*}
By (\ref{L3}), we have
\begin{align*}
\int_{\mathbb{S}^{d-1}}\int_{0}^{1}\frac{\vert \langle\theta,\nabla f(x+r\theta)\rangle-\langle\theta,\nabla f(x)\rangle\vert }{r}\dif r\dif\theta\leq&C\int_{0}^{1}(1-\ln r)\dif r\dif\theta\leq C,
\end{align*}
and by (\ref{L9}) and (\ref{decay1}), we have
\begin{align}\label{large1}
&\int_{\mathbb{S}^{d-1}}\int_{1}^{\infty}\frac{\vert \langle\theta,\nabla f(x+r\theta)\rangle\vert }{r}\dif r\dif\theta\nonumber\\
=&\int_{\mathbb{S}^{d-1}}\int_{1}^{\infty}\frac{\vert \langle r\theta,\nabla f(x+r\theta)\rangle\vert }{r^{2}}\dif r\dif\theta\nonumber\\
\leq&\int_{\mathbb{S}^{d-1}}\int_{1}^{\infty}\frac{\vert \langle(x+r\theta),\nabla f(x+r\theta)\rangle\vert +\vert \langle x,\nabla f(x+r\theta)\rangle\vert }{r^{2}}\dif r\dif\theta\nonumber\\
\leq&C\int_{\mathbb{S}^{d-1}}\int_{1}^{\infty}\frac{1+\vert x+r\theta\vert ^{\beta}+\vert x\vert }{r^{2}}\dif r\dif\theta\leq C\int_{1}^{\infty}\frac{1+\vert x\vert +r^{\beta}}{r^{2}}\leq C(1+\vert x\vert ).
\end{align}

When $\alpha=1$ and $\vert x\vert \geq1$, we have
\begin{align*}
&\int_{\mathbb{S}^{d-1}}\int_{0}^{\infty}\frac{\vert \langle\theta,\nabla f(x+r\theta)\rangle-\langle\theta,\nabla f(x){\bf 1}_{(0,1]}(r)\rangle\vert }{r}\dif r\dif\theta\\
=& \int_{\mathbb{S}^{d-1}}\!\int_{0}^{\vert x\vert }\frac{\vert \langle\theta,\nabla f(x+r\theta)\rangle-\langle\theta,\nabla f(x)\rangle\vert }{r}\dif r\dif\theta+\int_{\mathbb{S}^{d-1}}\!\int_{\vert x\vert }^{\infty}\frac{\vert \langle\theta,\nabla f(x+r\theta)\rangle\vert }{r}\dif r\dif\theta.
\end{align*}
By (\ref{L3}) and (\ref{L9}), we have
\begin{align*}
&\int_{\mathbb{S}^{d-1}}\int_{0}^{\vert x\vert }\frac{\vert \langle\theta,\nabla f(x+r\theta)\rangle-\langle\theta,\nabla f(x){\bf 1}_{(0,1]}(r)\rangle\vert }{r}\dif r\dif\theta\\
=&\int_{\mathbb{S}^{d-1}}\int_{0}^{1}\frac{\vert \langle\theta,\nabla f(x+r\theta)\rangle-\langle\theta,\nabla f(x)\rangle\vert }{r}\dif r\dif\theta
+\int_{\mathbb{S}^{d-1}}\int_{1}^{\vert x\vert }\frac{\vert \langle\theta,\nabla f(x+r\theta)\rangle\vert }{r}\dif r\dif\theta\\
\leq&C(1+\ln\vert x\vert ),
\end{align*}
and by (\ref{L9}) and (\ref{decay1}), we have
\begin{align*}
&\int_{\mathbb{S}^{d-1}}\int_{\vert x\vert }^{\infty}\frac{\vert \langle\theta,\nabla f(x+r\theta)\rangle\vert }{r}\dif r\dif\theta\nonumber\\
=&\int_{\mathbb{S}^{d-1}}\int_{\vert x\vert }^{\infty}\frac{\vert \langle r\theta,\nabla f(x+r\theta)\rangle\vert }{r^{2}}\dif r\dif\theta\nonumber\\
\leq&\int_{\mathbb{S}^{d-1}}\int_{\vert x\vert }^{\infty}\frac{\vert \langle(x+r\theta),\nabla f(x+r\theta)\rangle\vert +\vert \langle x,\nabla f(x+r\theta)\rangle\vert }{r^{2}}\dif r\dif\theta\nonumber\\
\leq&C\int_{\mathbb{S}^{d-1}}\int_{\vert x\vert }^{\infty}\frac{1+\vert x+r\theta\vert ^{\beta}+\vert x\vert }{r^{2}}\dif r\dif\theta\\
\leq& C\int_{\vert x\vert }^{\infty}\frac{1+\vert x\vert +r^{\beta}}{r^{2}}\leq C(1+\vert x\vert )\vert x\vert ^{-1}+\vert x\vert ^{\beta-1}\leq C.
\end{align*}

When $\alpha\in (0,1)$ and $\vert x\vert <1$, we have
\begin{align*}
&\int_{\mathbb{S}^{d-1}}\int_{0}^{\infty}\frac{\vert \langle\theta,\nabla f(x+r\theta)\rangle\vert }{r^\alpha}\dif r\dif\theta\\
=& \int_{\mathbb{S}^{d-1}}\int_{0}^{1}\frac{\vert \langle\theta,\nabla f(x+r\theta)\rangle\vert }{r^\alpha}\dif r\dif\theta+\int_{\mathbb{S}^{d-1}}\int_{1}^{\infty}\frac{\vert \langle\theta,\nabla f(x+r\theta)\rangle\vert }{r^\alpha}\dif r\dif\theta.
\end{align*}
By (\ref{firstb}), we have
\begin{align*}
\int_{\mathbb{S}^{d-1}}\int_{0}^{1}\frac{\vert \langle\theta,\nabla f(x+r\theta)\rangle\vert }{r^\alpha}\dif r\dif\theta\leq&C\int_{0}^{1}\frac{1}{r^\alpha}\dif r\dif\theta\leq C,
\end{align*}
and by (\ref{firstb}) and (\ref{decay}), we have
\begin{align}\label{large1}
&\int_{\mathbb{S}^{d-1}}\int_{1}^{\infty}\frac{\vert \langle\theta,\nabla f(x+r\theta)\rangle\vert }{r^\alpha}\dif r\dif\theta\nonumber\\
=&\int_{\mathbb{S}^{d-1}}\int_{1}^{\infty}\frac{\vert \langle r\theta,\nabla f(x+r\theta)\rangle\vert }{r^{\alpha+1}}\dif r\dif\theta\nonumber\\
\leq&\int_{\mathbb{S}^{d-1}}\int_{1}^{\infty}\frac{\vert \langle(x+r\theta),\nabla f(x+r\theta)\rangle\vert +\vert \langle x,\nabla f(x+r\theta)\rangle\vert }{r^{\alpha+1}}\dif r\dif\theta\nonumber\\
\leq&C\int_{\mathbb{S}^{d-1}}\int_{1}^{\infty}\frac{1+\vert x+r\theta\vert ^{\beta}+\vert x\vert }{r^{\alpha+1}}\dif r\dif\theta\leq C\int_{1}^{\infty}\frac{1+\vert x\vert +r^{\beta}}{r^{\alpha+1}}\leq C(1+\vert x\vert ).
\end{align}

When $\alpha\in(0,1)$ and $\vert x\vert \geq1$, we have
\begin{align*}
&\int_{\mathbb{S}^{d-1}}\int_{0}^{\infty}\frac{\vert \langle\theta,\nabla f(x+r\theta)\rangle\vert }{r^\alpha}\dif r\dif\theta\\
=& \int_{\mathbb{S}^{d-1}}\int_{0}^{\vert x\vert }\frac{\vert \langle\theta,\nabla f(x+r\theta)\rangle\vert }{r^\alpha}\dif r\dif\theta+\int_{\mathbb{S}^{d-1}}\int_{\vert x\vert }^{\infty}\frac{\vert \langle\theta,\nabla f(x+r\theta)\rangle\vert }{r^\alpha}\dif r\dif\theta.
\end{align*}
By (\ref{firstb}), we have
\begin{align*}
\int_{\mathbb{S}^{d-1}}\int_{0}^{\vert x\vert }\frac{\vert \langle\theta,\nabla f(x+r\theta)\rangle\vert }{r^\alpha}\dif r\dif\theta\leq&C\int_{0}^{\vert x\vert }\frac{1}{r^\alpha}\dif r\dif\theta\leq C\vert x\vert ^{1-\alpha},
\end{align*}
and by (\ref{firstb}) and (\ref{decay}), we have
\begin{align*}
&\int_{\mathbb{S}^{d-1}}\int_{\vert x\vert }^{\infty}\frac{\vert \langle\theta,\nabla f(x+r\theta)\rangle\vert }{r^\alpha}\dif r\dif\theta\nonumber\\
=&\int_{\mathbb{S}^{d-1}}\int_{\vert x\vert }^{\infty}\frac{\vert \langle r\theta,\nabla f(x+r\theta)\vert }{r^{\alpha+1}}\dif r\dif\theta\nonumber\\
\leq&\int_{\mathbb{S}^{d-1}}\int_{\vert x\vert }^{\infty}\frac{\vert \langle(x+r\theta),\nabla f(x+r\theta)\rangle\vert +\vert \langle x,\nabla f(x+r\theta)\rangle\vert }{r^{\alpha+1}}\dif r\dif\theta\nonumber\\
\leq&C\int_{\mathbb{S}^{d-1}}\int_{\vert x\vert }^{\infty}\frac{1+\vert x+r\theta\vert ^{\beta}+\vert x\vert }{r^{\alpha+1}}\dif r\dif\theta\\
\leq& C\int_{\vert x\vert }^{\infty}\frac{1+\vert x\vert +r^{\beta}}{r^{\alpha+1}}\leq C(1+\vert x\vert )\vert x\vert ^{-\alpha}+\vert x\vert ^{\beta-\alpha}\leq C(1+\vert x\vert ^{1-\alpha}).
\end{align*}
The proof is complete.
\end{proof}

\bl\label{alter1}
Let $\alpha \in(0,\frac{1}{2}]$ and $ h\in \mathcal H_\beta\cap\mathcal{F}_{\beta}$ with $\beta \in (0,\alpha)$. Let $f$ be
defined as \eqref{solution} and denote $\tilde{\beta}:=\max\{\beta,d-\gamma\}\in(0,\alpha)$. Then, for any $x\in\mathbb{R}^{d},$
\begin{align*}
\int_{\mathbb{S}^{d-1}}\int_{0}^{\infty}\frac{\vert \langle\theta,\nabla f(x+r\theta)\rangle\vert }{r^\alpha}\dif r\dif\theta\leq C(1+\vert x\vert ^{\tilde{\beta}-\alpha}).
\end{align*}
\el

\begin{proof}
When $\vert x\vert <1$, the proof is similar to the proof of Lemma \ref{alter}. When $\vert x\vert \geq1$, Proposition \ref{prop} implies that
\begin{align*}
&\int_{\mathbb{S}^{d-1}}\int_{0}^{\infty}\frac{\vert \langle\theta,\nabla f(x+r\theta)\rangle\vert }{r^\alpha}\dif r\dif\theta\\
\leq&C\int_{\mathbb{S}^{d-1}}\int_{0}^{\infty}\frac{\vert x+r\theta\vert ^{\tilde{\beta}-1}}{r^\alpha}\dif r\dif\theta\\
\leq&C\int_{0}^{\infty}\frac{\left\vert \vert x\vert -r\right\vert ^{\tilde{\beta}-1}}{r^\alpha}\dif r
=C\left[\int_{0}^{\vert x\vert }\frac{(\vert x\vert -r)^{\tilde{\beta}-1}}{r^\alpha}\dif r+\int_{\vert x\vert }^{\infty}\frac{(r-\vert x\vert )^{\tilde{\beta}-1}}{r^\alpha}\dif r\right].
\end{align*}
One can write
\begin{align*}
\int_{0}^{\vert x\vert }\frac{(\vert x\vert -r)^{\tilde{\beta}-1}}{r^\alpha}\dif r=&\int_{0}^{\frac{\vert x\vert }{2}}\frac{(\vert x\vert -r)^{\tilde{\beta}-1}}{r^\alpha}\dif r+\int_{\frac{\vert x\vert }{2}}^{\vert x\vert }\frac{(\vert x\vert -r)^{\tilde{\beta}-1}}{r^\alpha}\dif r\\
\leq&2^{1-\tilde{\beta}}\vert x\vert ^{\tilde{\beta}-1}\int_{0}^{\frac{\vert x\vert }{2}}\frac{1}{r^\alpha}\dif r+2^{\alpha}\vert x\vert ^{-\alpha}\int_{\frac{\vert x\vert }{2}}^{\vert x\vert }(\vert x\vert -r)^{\tilde{\beta}-1}\dif r\\
\leq& C\vert x\vert ^{\tilde{\beta}-\alpha},
\end{align*}
whereas
\begin{align*}
\int_{\vert x\vert }^{\infty}\frac{(r-\vert x\vert )^{\tilde{\beta}-1}}{r^\alpha}\dif r=\frac{1}{\tilde{\beta}}\int_{\vert x\vert }^{\infty}\frac{1}{r^{\alpha}}\dif(r-\vert x\vert )^{\tilde{\beta}}=&\frac{\alpha}{\tilde{\beta}}\int_{\vert x\vert }^{\infty}\frac{(r-\vert x\vert )^{\tilde{\beta}}}{r^{\alpha+1}}\dif r\\
\leq&\frac{\alpha}{\tilde{\beta}(\alpha-\beta)}\vert x\vert ^{\tilde{\beta}-\alpha}.
\end{align*}
The proof is complete.
\end{proof}

\subsection{Taylor-like expansion}
In order to prove the main Theorem, we shall make use of the following lemmas.

$\bullet$ \underline{{\bf $\alpha\in(1,2)$}} {\bf :}
\begin{lemma}\label{2case}
Consider $\alpha\in(1,2)$. Let $X$ be a d-dimensional random vector with density function $p_{X}(r)dr\nu(d\theta)$ and $Y$ be a $d$-dimensional random vector, which is independent of $X$. For any $a>0$ and $f$ is defined as \eqref{solution}, denote
\begin{align*}
T_{1}:=\Big\vert \mathbb{E}[\langle X,\nabla f(Y+aX)\rangle-\langle X,\nabla f(Y)\rangle]-\frac{\alpha^{2}}{d_{\alpha}}a^{\alpha-1}\mathbb{E}[\mathcal{L}^{\alpha,\nu}f(Y)]\Big\vert ,
\end{align*}
then, we have
\begin{align*}
T_{1}\leq C\left(a\int_{0}^{a^{-1}}r^{2}\big\vert \frac{\alpha}{r^{\alpha+1}}-p_{X}(r)\big\vert dr+\int_{a^{-1}}^{\infty}\big\vert \frac{\alpha}{r^{\alpha}}-rp_{X}(r)dr\big\vert \right).
\end{align*}
\end{lemma}
\begin{proof}
From Lemma \ref{properties}, we have
\begin{align*}
\frac{\alpha^{2}}{d_{\alpha}}a^{\alpha-1}\mathbb{E}[\mathcal{L}^{\alpha,\nu}f(Y)]
=&\alpha\mathbb{E}\big[\int_{\mathbb{S}^{d-1}}\int_{0}^{\infty}\frac{\langle r\theta,\nabla f(Y+ar\theta)\rangle-\langle r\theta,\nabla f(Y)\rangle}{r^{\alpha+1}}dr\nu(d\theta)\big],
\end{align*}
and
\begin{align*}
&\mathbb{E}[\langle X,\nabla f(Y+aX)\rangle-\langle X,\nabla f(Y)\rangle]\\
=&\mathbb{E}\big[\int_{\mathbb{S}^{d-1}}\int_{0}^{\infty}\left(\langle r\theta,\nabla f(Y+ar\theta)\rangle-\langle r\theta,\nabla f(Y)\rangle\right)p_{X}(r)dr\nu(d\theta)\big].
\end{align*}
These imply
\begin{align}\label{taylor1}
T_{1}\leq&\mathbb{E}\Big\vert \int_{\mathbb{S}^{d-1}}\int_{0}^{a^{-1}}\langle r\theta,\nabla f(Y+ar\theta)-\nabla f(Y)\rangle\big[\frac{\alpha}{r^{\alpha+1}}dr\nu(d\theta)-p_{X}(r)dr\nu(d\theta)\big]\Big\vert \nonumber\\
&+\mathbb{E}\Big\vert \int_{\mathbb{S}^{d-1}}\int_{a^{-1}}^{\infty}\langle r\theta,\nabla f(Y+ar\theta)-\nabla f(Y)\rangle\big[\frac{\alpha}{r^{\alpha+1}}dr\nu(d\theta)-p_{X}(r)dr\nu(d\theta)\big]\Big\vert \nonumber\\
:=&\mathcal{I}_{1}+\mathcal{I}_{2}.
\end{align}
Then, one can write by (\ref{second order}) that
\begin{align*}
\mathcal{I}_{1}\leq Ca\int_{\mathbb{S}^{d-1}}\int_{0}^{a^{-1}}r^{2}\big\vert \frac{\alpha}{r^{\alpha+1}}dr-p_{X}(r)dr\big\vert \nu(d\theta)
\leq Ca\int_{0}^{a^{-1}}r^{2}\big\vert \frac{\alpha}{r^{\alpha+1}}-p_{X}(r)\big\vert dr,
\end{align*}
whereas by (\ref{first order})
\begin{align*}
\mathcal{I}_{2}\leq C\int_{\mathbb{S}^{d-1}}\int_{a^{-1}}^{\infty}\big\vert \frac{\alpha}{r^{\alpha}}-rp_{X}(r)dr\big\vert \nu(d\theta)
\leq C\int_{a^{-1}}^{\infty}\big\vert \frac{\alpha}{r^{\alpha}}-rp_{X}(r)dr\big\vert ,
\end{align*}
the desired result follows.
\end{proof}

$\bullet$ \underline{{\bf $\alpha=1$}} {\bf :}
\bl\label{Lemregu2}
Consider $\alpha=1$ and $\int_{\mathbb{S}^{d-1}}\theta\nu(d\theta)=0.$ Let $X$ be a d-dimensional random vector with density function $p_{X}(r)dr\nu(d\theta)$ and suppose that $p_{X}(r)$ is non-increasing. Let $Y$ be a $d$-dimensional random vector, which is independent of $X$. For any $a>0$ and $f$ is defined as \eqref{solution}, denote
\begin{align*}
T_{2}:=\Big\vert \mathbb{E}[\langle X,\nabla f(Y+aX)\rangle-\langle X,\nabla f(Y){\bf 1}_{(0,1]}(a\vert X\vert )\rangle]-\frac{1}{d_{1}}\mathbb{E}[\mathcal{L}^{1,\nu}f(Y)]\Big\vert ,
\end{align*}
then, we have
\begin{align*}
T_{2}\leq C\Big(a\int_{0}^{a^{-1}}r^{2}\big(1-\log(ar)\big)\big\vert \frac{\alpha}{r^{2}}-p_{X}(r)\big\vert dr+a^{\beta-1}\int_{a^{-1}}^{\infty}t^{\beta}\big\vert d\big[\frac{\alpha}{t}-tp_{X}(t)\big]\big\vert \Big).
\end{align*}
\el
\begin{proof}
By the same argument as the proof of (\ref{taylor1}), we have
\begin{align}
T_{2}\leq&\mathbb{E}\Big\vert \int_{\mathbb{S}^{d-1}}\int_{0}^{a^{-1}}\langle r\theta,\nabla f(Y+ar\theta)-\nabla f(Y)\rangle]\big[\frac{1}{r^{2}}dr\nu(d\theta)-p_{X}(r)dr\nu(d\theta)\big]\Big\vert \nonumber\\
&+\mathbb{E}\Big\vert \int_{\mathbb{S}^{d-1}}\int_{a^{-1}}^{\infty}\langle r\theta,\nabla f(Y+ar\theta)\rangle\big[\frac{1}{r^{2}}dr\nu(d\theta)-p_{X}(r)dr\nu(d\theta)\big]\Big\vert \nonumber\\
:=&\mathcal{J}_{1}+\mathcal{J}_{2}.
\end{align}
On the one hand, (\ref{L3}) derives that
\begin{align*}
\mathcal{J}_{1}\leq&\mathbb{E}\Big[\int_{\mathbb{S}^{d-1}}\int_{0}^{a^{-1}}\vert \langle r\theta,\nabla f(Y+ar\theta)-\nabla f(Y)\rangle\vert \big\vert \frac{1}{r^{2}}dr-p_{X}(r)dr\big\vert \nu(d\theta)\Big]\\ \leq&Ca\int_{\mathbb{S}^{d-1}}\int_{0}^{a^{-1}}r^{2}\big(1-\log(ar)\big)\big\vert \frac{1}{r^{2}}-p_{X}(r)\big\vert dr\nu(d\theta)\\
\leq&Ca\int_{0}^{a^{-1}}r^{2}\big(1-\log(ar)\big)\big\vert \frac{1}{r^{2}}-p_{X}(r)\big\vert dr.
\end{align*}
On the other hand, noting that $p_{X}(r)$ is non-increasing and $\int_{0}^{\infty}p_{X}(r)dr<\infty,$ which imply $\lim_{r\rightarrow\infty}rp_{X}(r)=0.$ So we have by integration by parts that
\begin{align*}
\mathcal{J}_{2}=&\mathbb{E}\Big\vert \int_{\mathbb{S}^{d-1}}\int_{a^{-1}}^{\infty}\langle\theta,\nabla f(Y+ar\theta)\rangle\big[\frac{1}{r}-rp_{X}(r)\big]dr\nu(d\theta)\Big\vert \\
=&\mathbb{E}\Big\vert \int_{\mathbb{S}^{d-1}}\int_{a^{-1}}^{\infty}\langle\theta,\nabla f(Y+ar\theta)\rangle dr\int_{r}^{\infty}d\big[\frac{1}{t}-tp_{X}(t)\big]\nu(d\theta)\Big\vert \\
=&\mathbb{E}\Big\vert \int_{\mathbb{S}^{d-1}}\int_{a^{-1}}^{\infty}d\big[\frac{1}{t}-tp_{X}(t)\big]\int_{a^{-1}}^{t}\langle\theta,\nabla f(Y+ar\theta)\rangle dr\nu(d\theta)\Big\vert \\
=&a^{-1}\mathbb{E}\Big\vert \int_{\mathbb{S}^{d-1}}\int_{a^{-1}}^{\infty}\big(f(Y+at\theta)-f(Y+\theta)\big)d\big[\frac{1}{t}-tp_{X}(t)\big]\nu(d\theta)\Big\vert ,
\end{align*}
then we have by (\ref{L4})
\begin{align}\label{taylor2}
\mathcal{J}_{2}\leq&Ca^{\beta-1}\int_{\mathbb{S}^{d-1}}\int_{a^{-1}}^{\infty}t^{\beta}\big\vert d\big[\frac{1}{t}-tp_{X}(t)\big]\big\vert \nu(d\theta)
\leq Ca^{\beta-1}\int_{a^{-1}}^{\infty}t^{\beta}\big\vert d\big[\frac{1}{t}-tp_{X}(t)\big]\big\vert ,
\end{align}
the desired conclusion follows.
\end{proof}

$\bullet$ \underline{{\bf $\alpha\in(0,1)$}} {\bf :} For any $x\in\mathbb{R}^{d}$, we have
\begin{align*}
\int_{\mathbb{S}^{d-1}}\int_{0}^{1}\frac{\langle r\theta,\nabla f(x)\rangle}{r^{\alpha+1}}dr\nu(d\theta)=\frac{1}{1-\alpha}\int_{\mathbb{S}^{d-1}}\langle\theta,\nabla f(x)\rangle\nu(d\theta),
\end{align*}
which follows that
\begin{align}\label{alp1}
&\frac{1}{d_{\alpha}}\mathcal{L}^{\alpha}f(x)-\frac{1}{\alpha(1-\alpha)}\int_{\mathbb{S}^{d-1}}\langle\theta,\nabla f(Y)\rangle\nu(d\theta)\nonumber\\
=&\frac{1}{\alpha}\int_{\mathbb{S}^{d-1}}\int_{0}^{\infty}\frac{\langle r\theta,\nabla f(x+r\theta)\rangle-\langle r\theta,\nabla f(x){\bf 1}_{(0,1]}(r)\rangle}{r^{\alpha+1}}dr\nu(d\theta).
\end{align}
According to (\ref{alp1}), we have the following Taylor-like expansion.

\bl\label{random}\label{Lemregu1}
Consider $\alpha\in(0,1)$ and when $\alpha\in(0,\frac{1}{2}]$, we further assume $h\in\mathcal{F}_{\beta}$. Let $X$ be a d-dimensional random vector with density function $p_{X}(r)dr\nu(d\theta)$ and suppose that $p_{X}(r)$ is non-increasing. Let $Y$ be a $d$-dimensional random vector, which is independent of $X$. For any $a>0$ and $f$ is defined as \eqref{solution}, denote
\begin{align*}
T_{3}:=\Big\vert &\mathbb{E}[\langle X,\nabla f(Y+aX)\rangle-\langle X,\nabla f(Y){\bf 1}_{(0,1]}(a\vert X\vert )\rangle]\\
&-\frac{\alpha^{2}}{d_{\alpha}}a^{\alpha-1}\mathbb{E}[\mathcal{L}^{\alpha,\nu}f(Y)-\frac{d_{\alpha}}{\alpha(1-\alpha)}\int_{\mathbb{S}^{d-1}}\langle\theta,\nabla f(Y)\nu(d\theta)\rangle]\Big\vert ,
\end{align*}
then, we have
\begin{align*}
T_{3}\leq C\Big(a^{\alpha}\int_{0}^{a^{-1}}r^{\alpha+1}\big\vert \frac{\alpha}{r^{\alpha+1}}-p_{X}(r)\big\vert dr
+a^{\beta-1}\int_{a^{-1}}^{\infty}t^{\beta}\big\vert d\big[\frac{\alpha}{t^{\alpha}}-tp_{X}(t)\big]\big\vert \Big).
\end{align*}
\el
\begin{proof}
According to (\ref{alp1}), by the same argument as the proof of (\ref{taylor1}), we have
\begin{align*}
T_{3}\leq&\mathbb{E}\Big\vert \int_{\mathbb{S}^{d-1}}\int_{0}^{a^{-1}}[\langle r\theta,\nabla f(Y+ar\theta)-\nabla f(Y)\rangle]\big[\frac{\alpha}{r^{\alpha+1}}dr\nu(d\theta)-p_{X}(r)dr\nu(d\theta)\big]\Big\vert \nonumber\\
&+\mathbb{E}\Big\vert \int_{\mathbb{S}^{d-1}}\int_{a^{-1}}^{\infty}\langle r\theta,\nabla f(Y+ar\theta)\rangle\big[\frac{\alpha}{r^{\alpha+1}}dr\nu(d\theta)-p_{X}(r)dr\nu(d\theta)\big]\Big\vert \nonumber\\
:=&\mathcal{I}+\mathcal{II}.
\end{align*}
One can write by (\ref{L1}) that
\begin{align*}
\mathcal{I}\leq&\mathbb{E}\Big[\int_{\mathbb{S}^{d-1}}\int_{0}^{a^{-1}}\vert \langle r\theta,\nabla f(Y+ar\theta)\rangle-\langle r\theta,\nabla f(Y)\rangle\vert \big\vert \frac{\alpha}{r^{\alpha+1}}dr-p_{X}(r)dr\big\vert \nu(d\theta)\Big]\\
\leq&Ca^{\alpha}\int_{\mathbb{S}^{d-1}}\int_{0}^{a^{-1}}r^{\alpha+1}\big\vert \frac{\alpha}{r^{\alpha+1}}-p_{X}(r)\big\vert dr\nu(d\theta)\\
=&Ca^{\alpha}\int_{0}^{a^{-1}}r^{\alpha+1}\big\vert \frac{\alpha}{r^{\alpha+1}}-p_{X}(r)\big\vert dr,
\end{align*}
whereas by the same argument as the proof of (\ref{taylor2}),
\begin{align*}
\mathcal{II}\leq&Ca^{\beta-1}\int_{\mathbb{S}^{d-1}}\int_{a^{-1}}^{\infty}t^{\beta}\big\vert d\big[\frac{\alpha}{t^{\alpha}}-tp_{X}(t)\big]\big\vert \nu(d\theta)
=Ca^{\beta-1}\int_{a^{-1}}^{\infty}t^{\beta}\big\vert d\big[\frac{\alpha}{t^{\alpha}}-tp_{X}(t)\big]\big\vert ,
\end{align*}
the desired conclusion follows.
\end{proof}

\subsection{Truncation for random variable $X$}

In the case $\alpha\in(0,1],$ the random variable $X$ considered here satisfies $\mathbb{E}\vert X\vert ^{\alpha}=\infty$. Therefore, we need to truncate the random variable $X$.

\bl\label{trunc1}
Consider $\alpha\in(0,1]$ and when $\alpha=1$ we assume $\int_{\mathbb{S}^{d-1}}\theta\nu(d\theta)=0.$ Let $X$ be a d-dimensional random vector with density function $p_{X}(r)dr\nu(d\theta)$ and $f$ be defined as \eqref{solution}. Then for any $0<a<1$ and $z\in\mathbb{R}^{d},$ we have\\
1.) when $\alpha=1,$
\begin{align*}
\mathbb{E}\big\vert \mathcal{L}^{1,\nu}f(z)-\mathcal{L}^{1,\nu}f(z+aX)\big\vert \!\leq\! C\Big(\int_{a^{-1}}^{\infty}p_{X}(r)dr+a\int_{0}^{a^{-1}}r\big(1-\log(ar)\big)p_{X}(r)dr\Big).
\end{align*}
2.) when $\alpha\in(0,1),$
\begin{align*}
\mathbb{E}\big\vert \mathcal{L}^{\alpha,\nu}f(z)-\mathcal{L}^{\alpha,\nu}f(z+aX)\big\vert \leq C\Big(\int_{a^{-1}}^{\infty}p_{X}(r)dr+a^{\frac{1+\alpha}{2}}\int_{0}^{a^{-1}}r^{\frac{1+\alpha}{2}}p_{X}(r)dr\Big).
\end{align*}
\el
\begin{proof}
Observe
\begin{align*}
&\mathbb{E}\Big[\big\vert \mathcal{L}^{\alpha,\nu}f(z)-\mathcal{L}^{\alpha,\nu}f(z+aX)\big\vert \Big]\\
=&\mathbb{E}\Big[\big\vert \mathcal{L}^{\alpha,\nu}f(z)-\mathcal{L}^{\alpha,\nu}f(z+aX)\big\vert \big[{\bf 1}_{(a^{-1},\infty)}(\vert X\vert )+{\bf 1}_{((0,a^{-1}))}(\vert X\vert )\big]\Big]:=\mathrm{I}+\mathrm{II}.
\end{align*}
When $\alpha=1,$ one can write by (\ref{L7})
\begin{align*}
\mathrm{I}\leq C\mathbb{P}\big(\vert X\vert >a^{-1}\big)=C\int_{\mathbb{S}^{d-1}}\int_{a^{-1}}^{\infty}p_{X}(r)dr\nu(d\theta)\leq C\int_{a^{-1}}^{\infty}p_{X}(r)dr,
\end{align*}
whereas by (\ref{L8})
\begin{align*}
\mathrm{II}\leq Ca\mathbb{E}\big[\vert X\vert \big(1-\log(a\vert X\vert )\big){\bf 1}_{(0,a^{-1})}(\vert X\vert )\big]\leq Ca\int_{0}^{a^{-1}}r\big(1-\log(ar)\big)p_{X}(r)dr.
\end{align*}
When $\alpha\in(0,1),$ one can write by (\ref{L5})
\begin{align*}
\mathrm{I}\leq C_{\alpha,\beta}\mathbb{P}\big(\vert X\vert >a^{-1}\big)=C\int_{\mathbb{S}^{d-1}}\int_{a^{-1}}^{\infty}p_{X}(r)dr\nu(d\theta)=C\int_{a^{-1}}^{\infty}p_{X}(r)dr,
\end{align*}
whereas by (\ref{L6}) with $\eta=\frac{1+\alpha}{2}\in(\alpha,1)$
\begin{align*}
\mathrm{II}\leq Ca^{\frac{1+\alpha}{2}}\mathbb{E}\big[\vert X\vert ^{\frac{1+\alpha}{2}}{\bf 1}_{(0,a^{-1})}(\vert X\vert )\big]=Ca^{\frac{1+\alpha}{2}}\int_{0}^{a^{-1}}r^{\frac{1+\alpha}{2}}p_{X}(r)dr,
\end{align*}
the desired conclusion follows.
\end{proof}


\subsection{the proof of Theorem \ref{t:bound}}

Now, we are ready to use the Leave-one-out method to prove our second main result.

\begin{proof}[Proof of Theorem \ref{t:bound}]
By Eq. (\ref{e:steinsequation}), we have
\begin{align*}
\alpha\Big\vert \mathbb{E}\big[h(S_{n})\big]-\pi(h)\Big\vert =\Big\vert \mathbb{E}[\alpha\mathcal{L}^{\alpha,\nu}f(S_{n})-\langle S_{n},\nabla f(S_{n})\rangle]\Big\vert \leq\mathcal{N}_{1}+\mathcal{N}_{2}+\mathcal{N}_{3},
\end{align*}
where
\begin{align*}
\mathcal{N}_{1}=\frac{\alpha}{n}\sum_{i=1}^{n}\Big\vert \mathbb{E}\big[(\mathcal{L}^{\alpha,\nu}f)(S_{n}(i))-\mathbb{E}\big[(\mathcal{L}^{\alpha,\nu}f)(S_{n})\big]\Big\vert .
\end{align*}
If $\alpha\in(1,2)$,
\begin{align*}
\mathcal{N}_{2}=l_{n}^{-\frac{1}{\alpha}}\sum_{i=1}^{n}\Big\vert \mathbb{E}\big[\langle\eta_{n,i},\nabla f(S_{n}(i)+l_{n}^{-\frac{1}{\alpha}}\eta_{n,i})\rangle\big]&-\mathbb{E}\big[\langle\eta_{n,i},\nabla f(S_{n}(i))\rangle\big]\\
&-\frac{\alpha^{2}}{d_{\alpha}}l_{n}^{\frac{1-\alpha}{\alpha}}\mathbb{E}\big[(\mathcal{L}^{\alpha,\nu}f)(S_{n}(i))\big]\Big\vert
\end{align*}
\begin{align*}
\mathcal{N}_{3}=l_{n}^{-\frac{1}{\alpha}}\sum_{i=1}^{n}\left\vert \mathbb{E}\big[\eta_{n,i}\big]\right\vert \left\vert \mathbb{E}\Big[\nabla f(S_{n}(i))-\nabla f\big(S_{n}(i)+l_{n}^{-\frac{1}{\alpha}}\eta_{n,i}\big)\Big]\right\vert ;
\end{align*}
If $\alpha=1$
\begin{align*}
\mathcal{N}_{2}=l_{n}^{-1}\sum_{i=1}^{n}\Big\vert \mathbb{E}\big[\langle\eta_{n,i},\nabla f(S_{n}(i)+l_{n}^{-1}\eta_{n,i})\rangle\big]&-\mathbb{E}\big[\langle\eta_{n,i},\nabla f(S_{n}(i)){\bf 1}_{(0,l_{n}]}(\vert \eta_{n,i}\vert )\rangle\big]\\
&-\frac{1}{d_{1}}\mathbb{E}\big[(\mathcal{L}^{1,\nu}f)(S_{n}(i))\big]\Big\vert ,
\end{align*}
\begin{align*}
\mathcal{N}_{3}=l_{n}^{-1}\sum_{i=1}^{n}\big\vert \mathbb{E}\big[\eta_{n,i}{\bf 1}_{(0,l_{n}]}(\vert \eta_{n,i}\vert )\big]\big\vert \Big\vert \mathbb{E}\Big[\nabla f(S_{n}(i))-\nabla f\big(S_{n}(i)+l_{n}^{-1}\eta_{n,i}\big)\Big]\Big\vert ;
\end{align*}
If $\alpha\in(0,1)$,
\begin{align*}
\mathcal{N}_{2}=l_{n}^{-\frac{1}{\alpha}}\sum_{i=1}^{n}\Big\vert &\mathbb{E}\big[\langle\eta_{n,i},\nabla f(S_{n}(i)+l_{n}^{-\frac{1}{\alpha}}\eta_{n,i})\rangle\big]-\mathbb{E}\big[\langle\eta_{n,i},\nabla f(S_{n}(i)){\bf 1}_{(0,l_{n}^{\frac{1}{\alpha}}]}(\vert \eta_{n,i}\vert )\rangle\big]\\
&-\frac{\alpha^{2}}{d_{\alpha}}l_{n}^{\frac{1-\alpha}{\alpha}}\mathbb{E}\big[(\mathcal{L}^{\alpha,\nu}f)(S_{n}(i))-\frac{d_{\alpha}}{\alpha(1-\alpha)}\int_{\mathbb{S}^{d-1}}\langle\theta,\nabla f(S_{n}(i))\rangle\nu(d\theta)\big]\Big\vert ,
\end{align*}
\begin{align*}
\mathcal{N}_{3}=l_{n}^{-\frac{1}{\alpha}}\sum_{i=1}^{n}\Big\vert \frac{\alpha}{1-\alpha}l_{n}^{\frac{1-\alpha}{\alpha}}&\mathbb{E}\big[\int_{\mathbb{S}^{d-1}}\langle\theta,\nabla f(S_{n}(i))\rangle\nu(d\theta)\big]\\
&-\mathbb{E}\big[\langle\eta_{n,i},\nabla f(S_{n}(i)){\bf 1}_{(0,l_{n}^{\frac{1}{\alpha}}]}(\vert \eta_{n,i}\vert )\rangle\big]\Big\vert .
\end{align*}

1) When $\alpha\in(1,2)$, we have by (\ref{e:fracder})
\begin{align*}
\mathcal{N}_{1}\leq Cl_{n}^{-\frac{2}{\alpha}}\sum_{i=1}^{n}\mathbb{E}\vert \eta_{n,i}\vert ^{2-\alpha}
\end{align*}
and Lemma \ref{2case} implies that
\begin{align*}
\mathcal{N}_{2}\leq C\sum_{i=1}^{n}\left(l_{n}^{-\frac{2}{\alpha}}\int_{0}^{l_{n}^{\frac{1}{\alpha}}}r^{2}\big\vert \frac{\alpha}{r^{\alpha+1}}-p_{\eta_{n,i}}(r)\big\vert dr
+l_{n}^{-\frac{1}{\alpha}}\int_{l_{n}^{\frac{1}{\alpha}}}^{\infty}\big\vert \frac{\alpha}{r^{\alpha}}-rp_{\eta_{n,i}}(r)dr\big\vert \right).
\end{align*}
For the third term, one can derive from (\ref{second order}) that
\begin{align*}
\mathcal{N}_{3}\leq Cl_{n}^{-\frac{2}{\alpha}}\sum_{i=1}^{n}\left(\mathbb{E}\vert \eta_{n,i}\vert \right)^{2}.
\end{align*}

2) When $\alpha=1$ and $\int_{\mathbb{S}^{d-1}}\theta\nu(d\theta)=0,$ we have by Lemma \ref{trunc1}
\begin{align*}
\mathcal{N}_{1}\leq\frac{C}{n}\sum_{i=1}^{n}\Big(l_{n}^{-1}\int_{0}^{l_{n}}r\big(1-\log(l_{n}^{-1}r)\big)p_{\eta_{n,i}}(r)dr+\int_{l_{n}}^{\infty}p_{\eta_{n,i}}(r)dr\Big).
\end{align*}
By Lemma \ref{Lemregu2}, we have
\begin{align*}
\mathcal{N}_{2}\leq Cl_{n}^{-1}\sum_{i=1}^{n}\Big(&l_{n}^{-1}\int_{0}^{l_{n}}r^{2}\big(1-\log(l_{n}^{-1}r)\big)\vert \frac{1}{r^{2}}-p_{\eta_{n,i}}(r)\vert dr\\
&+l_{n}^{1-\beta}\int_{l_{n}}^{\infty}t^{\beta}\big\vert d[\frac{1}{t}-tp_{\eta_{n,i}}(t)]\big\vert \Big).
\end{align*}
In addition, noticing that $\int_{\mathbb{S}^{d-1}}\theta\nu(d\theta)=0,$ we have $\mathcal{N}_{3}=0.$

3) When $\alpha\in(0,1),$ we have by Lemma \ref{trunc1},
\begin{align*}
\mathcal{N}_{1}\leq C\frac{\alpha}{n}\sum_{i=1}^{n}\Big(l_{n}^{-\frac{\alpha+1}{2\alpha}}\int_{0}^{l_{n}^{\frac{1}{\alpha}}}r^{\frac{\alpha+1}{2}}p_{\eta_{n,i}}(r)dr+\int_{l_{n}^{\frac{1}{\alpha}}}^{\infty}p_{\eta_{n,i}}(r)dr\Big).
\end{align*}
By Lemma \ref{Lemregu1}, we have
\begin{align*}
\mathcal{N}_{2}\leq Cl_{n}^{-\frac{1}{\alpha}}\sum_{i=1}^{n}\Big(l_{n}^{-1}\int_{0}^{l_{n}^{\frac{1}{\alpha}}}r^{\alpha+1}\big\vert& \frac{\alpha}{r^{\alpha+1}}-p_{\eta_{n,i}}(r)\big\vert dr\\
&+l_{n}^{\frac{1-\beta}{\alpha}}\int_{l_{n}^{\frac{1}{\alpha}}}^{\infty}t^{\beta}\big\vert d[\frac{\alpha}{t^{\alpha}}-tp_{\eta_{n,i}}(t)]\big\vert \Big).
\end{align*}
In addition, we have
\begin{align*}
\mathbb{E}\big[\langle\eta_{n,i},\nabla f(S_{n}(i)){\bf 1}_{(0,l_{n}^{\frac{1}{\alpha}}]}(\vert \eta_{n,i}\vert )\rangle\big]=\int_{\mathbb{S}^{d-1}}\int_{0}^{l_{n}^{\frac{1}{\alpha}}}\langle r\theta,\nabla f(S_{n}(i))\rangle p_{\eta_{n,i}}(r)dr\nu(d\theta),
\end{align*}
which implies that
\begin{align*}
\mathcal{N}_{3}\leq&l_{n}^{-\frac{1}{\alpha}}\sum_{i=1}^{n}\Big\vert \mathbb{E}\Big[\int_{\mathbb{S}^{d-1}}\langle\theta,\nabla f(S_{n}(i))\rangle\nu(d\theta)\Big]\Big\vert \Big\vert \frac{\alpha}{1-\alpha}l_{n}^{\frac{1-\alpha}{\alpha}}-\int_{0}^{l_{n}^{\frac{1}{\alpha}}}rp_{\eta_{n,i}}(r)dr\Big\vert \\
\leq&\alpha l_{n}^{-\frac{1}{\alpha}}\sum_{i=1}^{n}\Big\vert \int_{\mathbb{S}^{d-1}}\theta\nu(d\theta)\Big\vert \Big\vert \frac{\alpha}{1-\alpha}l_{n}^{\frac{1-\alpha}{\alpha}}-\int_{0}^{l_{n}^{\frac{1}{\alpha}}}rp_{\eta_{n,i}}(r)dr\Big\vert .
\end{align*}
Combining all of above, the desired conclusion follows.
\end{proof}

\section{Example: $\nu-$Paretian distribution}\label{s:example}

In \cite{D-N}, Davydov and Nagaev considered the Pareto distribution $\xi$, that is, the density of the random variable $\xi$ is
\begin{equation}\label{Pareto distribution}
p_{\xi}(u)=
\begin{cases}
\alpha u^{-1-\alpha} \qquad if\quad u\geq1,\\
0 \qquad\qquad\quad if \quad u<1.
\end{cases}
\end{equation}
It is convenient to adhere the following definition.
\bd\label{pd}
We call a distribution $\nu-$Paretian if it corresponds to a random vector $\tau$ admitting the representation $\xi\varepsilon,$ where $\xi$ and $\varepsilon$ are independent, $\xi$ has the density (\ref{Pareto distribution}) while $\varepsilon$ is a random unit vector satisfying
 \begin{equation}\label{d-1 distribution}
P(\varepsilon\in E)=\nu(E),
\end{equation}
where $E\in\mathcal{B}_{\mathbb{S}^{d-1}}$, the Borel sets of $\mathbb{S}^{d-1}.$
\ed
In \cite{D-N}, the authors assumed that $\nu$ is symmetric and
$$
m_{\nu}=\min_{e\in\mathbb{S}^{d-1}}\Sigma_{\alpha}(e,\nu)>0,
$$
where $\Sigma_{\alpha}(e,\nu)=\int_{\mathbb{S}^{d-1}}\vert \langle e, \theta\rangle\vert ^{\alpha}\nu(\dif \theta).$ That means the $v-$Paretian distribution is strictly $d$-dimensional. Consider a sequence of i.i.d. random vectors such that
$$
\tau_{i}=^{d}\tau, \qquad i=1,2,\cdots .
$$
Set
\begin{align}\label{sum}
T_{n}=n^{-1/\alpha}\sum_{i=1}^{n}\tau_{i}.
\end{align}
Then, \cite{D-N} gave the following approximation of multidimensional stable law:
\bt\label{pthm}
\cite[Theorem 3.2]{D-N} Let $T_{n}$ be defined by (\ref{sum}). If the underlying distribution is $\nu-$Paretian then as $n\rightarrow\infty$
$$
\sup_{A\in \mathcal{B}(\mathbb{R}^{d})}\vert \mathbb{P}(T\in A)-\mathbb{P}(T_{n}\in A)\vert =\mathbf{O}(n^{-\delta}),
$$
where $\delta=\frac{\min(\alpha,2-\alpha)}{d+\alpha}$ and $T$ has the stable distribution determined by the characteristic function $E\e^{i\langle \lambda,T\rangle}=\exp(-\frac{\alpha}{d_{\alpha}}\vert \lambda\vert ^{\alpha}\Sigma_{\alpha}(\e_{\lambda}, \nu)),$ $\lambda\in\mathbb{R}^{d},d\geq1.$
\et

According to Theorem \ref{t:bound}, here we can consider the more general $\nu$ (see the assumptions in Lemma \ref{l:heatkernel}) and obtain a better convergence rate in Wasserstein(-type) distance.

\begin{theorem}
Keep the same assumptions as in Lemma \ref{l:heatkernel}. Set
$$
\zeta_{n,i}=\big(\frac{\alpha}{d_{\alpha}}\big)^{-\frac{1}{\alpha}}\frac{\tau_{i}}{n^{\frac{1}{\alpha}}}$$
and
\begin{equation*}
S_{n}=
\begin{cases}
\zeta_{n,1}-\mathbb{E}\zeta_{n,1}+\cdots+\zeta_{n,n}-\mathbb{E}\zeta_{n,n},\quad &\alpha\in(1,2),\\
\zeta_{n,1}-\mathbb{E}\zeta_{n,1}{\bf 1}_{(0,1]}(\vert \zeta_{n,1}\vert )+\cdots+\zeta_{n,n}-\mathbb{E}\zeta_{n,n}{\bf 1}_{(0,1]}(\vert \zeta_{n,n}\vert ), &\alpha=1,\\
\zeta_{n,1}+\zeta_{n,2}+\cdots+\zeta_{n,n},\quad &\alpha\in(0,1).
\end{cases}
\end{equation*}
Then,
\begin{align*}
d_{W}\big(\mathcal{L}(S_{n}),\pi\big)\leq Cn^{\frac{\alpha-2}{\alpha}}, \quad \alpha\in(1,2),
\end{align*}
and for any $\beta\in(0,\alpha)$,
\begin{align*}
d_{W_{\beta}}\big(\mathcal{L}(S_{n}),\mu\big)\leq C
\begin{cases}
n^{-1}(\log n)^{2},\quad &\alpha=1,\\
n^{-1}+\Big\vert \int_{\mathbb{S}^{d-1}}\theta\nu(d\theta)\Big\vert n^{\frac{\alpha-1}{\alpha}},\quad &\alpha\in(\frac{1}{2},1),
\end{cases}
\end{align*}
\begin{align*}
\sup_{h\in\mathcal{H}_{\beta}\cap\mathcal{F}_{\beta}}\big\vert \mathbb{E}h(S_{n})-\mu(h)\big\vert \leq C\Big(n^{-1}+\Big\vert \int_{\mathbb{S}^{d-1}}\theta\nu(d\theta)\Big\vert n^{\frac{\alpha-1}{\alpha}}\Big), \quad \alpha\in(0,\frac{1}{2}].
\end{align*}
\end{theorem}

\begin{proof}
By definition \ref{pd}, we obtain
\begin{align*}
p_{\tau_{i}}(r)dr\nu(d\theta)=
\begin{cases}
\frac{\alpha}{r^{\alpha+1}}dr\nu(d\theta),\quad &r\geq1,\\
\quad 0,\qquad\qquad &r<1.
\end{cases}
\end{align*}
Let $\zeta_{n,i}=l_{n}^{-1/\alpha}\tau_{i}$ and $\eta_{n,i}=l_{n}^{1/\alpha}\zeta_{n,i}=\tau_{i},$ it follows that
\begin{align*}
p_{\eta_{n,i}}(r)dr\nu(d\theta)=
\begin{cases}
\frac{\alpha}{r^{\alpha+1}}dr\nu(d\theta),\quad &r\geq1,\\
\quad 0,\qquad\qquad &r<1.
\end{cases}
\end{align*}
According to Theorem \ref{t:bound},\\
i) When $\alpha\in(1,2)$, we have
\begin{align*}
n^{-\frac{2}{\alpha}}\mathbb{E}\vert \eta_{n,i}\vert ^{2-\alpha}+n^{-\frac{2}{\alpha}}\big(\mathbb{E}\vert \eta_{n,i}\vert \big)^{2}
=n^{-\frac{2}{\alpha}}\mathbb{E}\vert \tau_{i}\vert ^{2-\alpha}+n^{-\frac{2}{\alpha}}\big(\mathbb{E}\vert \tau_{i}\vert \big)^{2}\leq Cn^{-\frac{2}{\alpha}},
\end{align*}
\begin{align*}
n^{-\frac{2}{\alpha}}\int_{0}^{l_{n}^{\frac{1}{\alpha}}}r^{2}\big\vert \frac{\alpha}{r^{\alpha+1}}-p_{\eta_{n,i}}(r)\big\vert dr
=n^{-\frac{2}{\alpha}}\int_{0}^{1}\frac{\alpha}{r^{\alpha-1}}dr=\frac{\alpha}{2-\alpha}n^{-\frac{2}{\alpha}}
\end{align*}
and
\begin{align*}
n^{-\frac{1}{\alpha}}\int_{l_{n}^{\frac{1}{\alpha}}}^{\infty}\big\vert \frac{\alpha}{r^{\alpha}}-rp_{\eta_{n,i}}(r)dr\big\vert =0.
\end{align*}
These inequalities imply $d_{W}\big(\mathcal{L}(S_{n}),\pi\big)\leq n^{\frac{\alpha-2}{\alpha}}$.\\
ii) When $\alpha=1$ and $\int_{\mathbb{S}^{d-1}}\theta\nu(d\theta)=0,$ we have
\begin{align*}
n^{-2}\int_{0}^{l_{n}}r\big(1-\log(l_{n}^{-1}r)\big)p_{\eta_{n,i}}(r)dr\leq Cn^{-2}\big(1+\log n+(\log n)^{2}\big),
\end{align*}
\begin{align*}
n^{-\beta}\int_{l_{n}}^{\infty}t^{\beta}\big\vert d[\frac{1}{t}-tp_{\eta_{n,i}}(t)]\big\vert =0,\quad n^{-1}\int_{l_{n}}^{\infty}p_{\eta_{n,i}}(r)dr=n^{-2},
\end{align*}
and
\begin{align*}
n^{-2}\int_{0}^{l_{n}}r^{2}\big(2-\log(l_{n}^{-1}r)\big)\vert \frac{1}{r^{2}}-p_{\eta_{n,i}}(r)\vert dr\leq Cn^{-2}(1+\log n).
\end{align*}
Hence, we have
\begin{align*}
d_{W_{\beta}}\big(\mathcal{L}(S_{n}),\mu\big)\leq Cn^{-1}(\log n)^{2}.
\end{align*}
(iii) When $\alpha\in(0,1),$ we have
\begin{align*}
n^{-\frac{3\alpha+1}{2\alpha}}\int_{0}^{l_{n}^{\frac{1}{\alpha}}}r^{\frac{\alpha+1}{2}}p_{\eta_{n,i}}(r)dr+n^{-1}\int_{l_{n}^{\frac{1}{\alpha}}}^{\infty}p_{\eta_{n,i}}(r)dr\leq Cn^{-2},
\end{align*}
\begin{align*}
n^{-\frac{\beta}{\alpha}}\int_{l_{n}^{\frac{1}{\alpha}}}^{\infty}t^{\beta}\big\vert d[\frac{\alpha}{t^{\alpha}}-tp_{\eta_{n,i}}(t)]\big\vert =0,
\end{align*}
\begin{align*}
n^{-\frac{1+\alpha}{\alpha}}\int_{0}^{l_{n}^{\frac{1}{\alpha}}}r^{\alpha+1}\big\vert \frac{\alpha}{r^{\alpha+1}}-p_{\eta_{n,i}}(r)\big\vert dr=\alpha n^{-\frac{1+\alpha}{\alpha}},
\end{align*}
and
\begin{align*}
\mathcal{R}_{n,\alpha,i}\leq\frac{\alpha}{1-\alpha}\Big\vert \int_{\mathbb{S}^{d-1}}\theta\nu(d\theta)\Big\vert n^{-\frac{1}{\alpha}}.
\end{align*}
Therefore, one can derive that
\begin{align*}
\sup_{h\in\mathcal{H}_{\beta}\cap\mathcal{F}_{\beta}}\big\vert \mathbb{E}h(S_{n})-\mu(h)\big\vert \leq C\Big(n^{-1}+\Big\vert \int_{\mathbb{S}^{d-1}}\theta\nu(d\theta)\Big\vert n^{\frac{\alpha-1}{\alpha}}\Big), \quad \alpha\in(0,\frac{1}{2}]
\end{align*}
and
\begin{align*}
d_{W}\big(\mathcal{L}(S_{n}),\mu\big)\leq C\Big(n^{-1}+\Big\vert \int_{\mathbb{S}^{d-1}}\theta\nu(d\theta)\Big\vert n^{\frac{\alpha-1}{\alpha}}\Big), \quad \alpha\in(\frac{1}{2},1).
\end{align*}
The proof is complete.
\end{proof}

\br
Let us compare our result with the known results in literatures. When the spectral measure $\nu$ is symmetric, the authors of \cite{D-N} obtained a rate $n^{-\frac{\alpha}{d+\alpha}}$ for $d$ dimensional stable law in total variation distance and conjectured that the rate can be improved to $n^{-\min\{1,\frac{2-\alpha}{\alpha}\}}$ in $L^{1}$ or total variation distance. Our results gives a positive answer to their conjecture for the Wasserstein(-type) distance.
\er

\begin{appendices}

\section{Some auxiliary estimates}\label{amoment}

\subsection{Moment estimate}

\begin{lemma}\label{moment}
Let $(Z_{t})_{t\geq0}$ be the strictly $\alpha$-stable L$\acute{e}$vy process with characteristic function $\widehat{\pi}$, which is defined in \eqref{e:fourierpi}. Suppose that the assumption of Lemma \ref{l:heatkernel} holds. Then for any $\beta\in(0,\alpha)$, there exists a constant $C>0$ such that
\begin{align*}
\mathbb{E}\vert Z_{1}\vert ^{\beta}<C.
\end{align*}
\end{lemma}

\begin{proof}
When $d=\gamma$, for any $\beta\in(0,\alpha)$, \eqref{est:p} implies
\begin{align*}
\mathbb{E}\vert Z_{1}\vert ^{\beta}=\int_{\mathbb{R}^{d}}\vert x\vert ^{\beta}p(x)dx\leq& \int_{\mathbb{R}^{d}}\frac{\vert x\vert ^{\beta}}{(1+\vert x\vert )^{\alpha+d}}dx\\
=&\int_{\mathbb{S}^{d-1}}\int_{0}^{1}\frac{r^{\beta}r^{d-1}}{(1+r)^{\alpha+d}}drd\theta+\int_{\mathbb{S}^{d-1}}\int_{1}^{\infty}\frac{r^{\beta}r^{d-1}}{(1+r)^{\alpha+d}}drd\theta\\
\leq&C\left(\int_{0}^{1}r^{\beta+d-1}dr+\int_{1}^{\infty}r^{\beta-\alpha-1}dr\right)\leq C.
\end{align*}
For the general case, since
\begin{align*}
\int_{\mathbb{S}^{d-1}}\int_{1}^{\infty}\frac{(r\vee1)^{\beta}}{r^{\alpha+1}}dtd\theta=\int_{1}^{\infty}\frac{r^{\beta}}{r^{\alpha+1}}dr=\frac{1}{\alpha-\beta},
\end{align*}
according to \cite[Theorem 25.3, Proposition 25.4 (ii) and (iii)]{S}, there exists a constants $C>0$ such that
\begin{align*}
\mathbb{E}\left(\vert Z_{1}\vert \vee1\right)^{\beta}\leq C.
\end{align*}
Hence, the desired result follow from the fact $\vert x\vert ^{\beta}\leq(\vert x\vert \vee1)^{\beta}$ for any $x\in\mathbb{R}^{d}$.
\end{proof}

\subsection{Heat Kernel Estimates of Rotationally Invariant $\alpha$-stable L$\acute{e}$vy process}

Let $p(t,x)$ be the transition probability density of rotationally invariant $\alpha$-stable process $Z_{t},$ which has characteristic function $\e^{-t\vert \lambda\vert ^{\alpha}}.$ It is well known that
\begin{align*}
p(t,x)=t^{-d/\alpha}p(1,t^{-1/\alpha}x),\qquad t>0,\quad x\in\mathbb{R}^{d}.
\end{align*}
Then, we have the following estimates:
\bl\label{abounds}
Let $p(x)$ be the probability density of $Z_{1},$ we have
\begin{align*}
p(x)\leq 2^{-d+1}\pi^{-\frac{d}{2}}\frac{\Gamma(d/\alpha)}{\alpha\Gamma(d/2)},\qquad p(x)\leq\frac{\alpha2^{\alpha-1}\sin\frac{\alpha\pi}{2}\Gamma((d+\alpha)/2)\Gamma(\alpha/2)}{\pi^{d/2+1}\vert x\vert ^{d+\alpha}},
\end{align*}
\begin{align}\label{estimate2}
\vert \nabla p(x)\vert \leq 2\pi\vert x\vert p_{(d+2)}(\tilde{x})
\end{align}
and
\begin{align}\label{estimate3}
\| \nabla^{2}p(x)\| _{{\rm op}}\leq2\pi p_{(d+2)}(1,\tilde{x})+4\pi^{2}\vert x\vert ^{2}p_{(d+4)}(\hat{x}),
\end{align}
where $\tilde{x}\in\mathbb{R}^{d+2}$ satisfies $\vert \tilde{x}\vert =\vert x\vert $, $p_{d+2}(\tilde{x})$ is the density of the rotationally invariant $\alpha$-stable process $Z_{1}$ in dimension $d+2$, $\hat{x}\in\mathbb{R}^{d+4}$ satisfies $\vert \hat{x}\vert =\vert x\vert $ and $p_{d+4}(\hat{x})$ is the density of the rotationally invariant $\alpha$-stable process $Z_{1}$ in dimension $d+4.$
\el

\begin{proof}
By the \cite[Proposition 2.3 (XII)]{S}, we have
\begin{align*}
p(x)=(2\pi)^{-d}\int_{\mathbb{R}^{d}}\e^{i\langle x,\lambda\rangle}\e^{-\vert \lambda\vert ^{\alpha}}\dif t=(2\pi)^{-d}\int_{\mathbb{R}^{d}}\cos(\langle x,\lambda\rangle)\e^{-\vert \lambda\vert ^{\alpha}}\dif \lambda,
\end{align*}
since $\vert \cos(\langle x,\lambda\rangle)\vert \leq1,$ we have
\begin{align*}
p(x)&\leq(2\pi)^{-d}\int_{\mathbb{R}^{d}}\e^{-\vert \lambda\vert ^{\alpha}}\dif \lambda\\
&=(2\pi)^{-d}\int_{\mathbb{S}^{d-1}}\int_{0}^{\infty}r^{d-1}\e^{-r^{\alpha}}\dif r\dif \theta\\
&=(2\pi)^{-d}V(\mathbb{S}^{d-1})\int_{0}^{\infty}r^{d-1}\e^{-r^{\alpha}}\dif r\\
&=(2\pi)^{-d}V(\mathbb{S}^{d-1})\frac{1}{\alpha}\int_{0}^{\infty}y^{\frac{d}{\alpha}-1}\e^{-y}\dif y\\
&=(2\pi)^{-d}V(\mathbb{S}^{d-1})\frac{\Gamma(d/\alpha)}{\alpha},
\end{align*}
where the last second equality is by taking $y=r^{\alpha}.$ Recall $V(\mathbb{S}^{d-1})=\frac{2\pi^{d/2}}{\Gamma(d/2)}$, we have
\begin{align*}
p(x)\leq2^{-d+1}\pi^{-\frac{d}{2}}\frac{\Gamma(d/\alpha)}{\alpha\Gamma(d/2)}.
\end{align*}
Using the Fourier inversion theorem for radial functions \cite[(2.1)]{B-G}, we have
\begin{align*}
p(x)=(2\pi)^{-\frac{d}{2}}\vert x\vert ^{-\frac{d}{2}+1}\int_{0}^{\infty}\e^{-t^{\alpha}}t^{\frac{d}{2}}J_{(d-2)/2}(\vert x\vert t)\dif t,
\end{align*}
where $J_{m}$ is the Bessel function of first kind of order $m.$ Then, let $r=\vert x\vert t$ in the above integral term, we have
\begin{align*}
p(x)=(2\pi)^{-\frac{d}{2}}\vert x\vert ^{-d}\int_{0}^{\infty}\e^{-(\frac{r}{\vert x\vert })^{\alpha}}r^{\frac{d}{2}}J_{(d-2)/2}(r)\dif r
\end{align*}
From \cite[section 7.2.8 (50)]{B-E}, we have $\frac{\partial}{\partial t}(t^{m}J_{m}(t))=t^{m}J_{m-1}(t).$ Hence, we use integration by parts
\begin{align*}
p(x)&=\alpha(2\pi)^{-\frac{d}{2}}\vert x\vert ^{-d-\alpha}\int_{0}^{\infty}\e^{-(\frac{r}{\vert x\vert })^{\alpha}}r^{\frac{d}{2}+\alpha-1}J_{d/2}(r)\dif r\\
&\leq\alpha(2\pi)^{-\frac{d}{2}}\vert x\vert ^{-d-\alpha}\int_{0}^{\infty}r^{\frac{d}{2}+\alpha-1}J_{d/2}(r)\dif r\\
&=\frac{\alpha2^{\alpha-1}\sin\frac{\alpha\pi}{2}\Gamma((d+\alpha)/2)\Gamma(\alpha/2)}{\pi^{d/2+1}\vert x\vert ^{d+\alpha}},
\end{align*}
where the last equality comes from the proof of \cite[Theorem 2.1]{B-G}.

Furthermore, from \cite[(11)]{B-J}, we have $\nabla p(x)=-2\pi x\hat{p}_{(d+2)}(\tilde{x})$, so
\begin{align*}
\vert \nabla p(x)\vert =2\pi\vert x\vert p_{(d+2)}(\tilde{x}),
\end{align*}
and by the same argument as the proof of \cite[(11)]{B-J}, we can also obtain
\begin{align*}
\| \nabla^{2}p(x)\| _{{\rm op}}\leq2\pi p_{(d+2)}(\tilde{x})+4\pi^{2}\vert x\vert ^{2}p_{(d+4)}(\hat{x}),
\end{align*}
the desired conclusions follow.
\end{proof}

\section{Proof of Proposition \ref{p:solution}}\label{a:solve_Stein}

\begin{lemma}\label{A1}
Let $(Q_{t})_{t\geq0}$ be a Markovian semigroup with transition density  $q(t,x,y)=p_{1-e^{-t}}(y-e^{-\frac{t}{\alpha}}x)$.  Then for any $h\in Lip(1)$ in the case $\alpha>1$ and $h\in\mathcal{H}_{\beta}$ for some $\beta<\alpha$ in the case $\alpha\leq1$, we have
\begin{align*}
\partial_{t} Q_{t}h(x)=\mathcal{A}^{\alpha,\nu}Q_{t}h(x).
\end{align*}
\end{lemma}
\begin{proof}
Recall that $q(t,x,y)=p_{1-e^{-t}}(y-e^{-\frac{t}{\alpha}}x)$ and $s(t)=1-e^{-t}.$ Then
\begin{align*}
  \left\vert \frac{\partial}{\partial t}q(t,x,y)\right\vert  &= \left\vert e^{-t}\frac{\partial}{\partial s(t)}p_{1-e^{-t}}(y-e^{-\frac{t}{\alpha}}x)+\alpha^{-1}e^{-\frac{t}{\alpha}}x\frac{\partial}{\partial y}p_{1-e^{-t}}(y-e^{-\frac{t}{\alpha}}x)\right\vert \\
   & \le \frac{C(1-e^{-t})^{(\gamma-d)/\alpha}}{((1-e^{-t})^{1/\alpha}+\vert y-e^{-\frac{t}{\alpha}}x\vert )^{\alpha+\ga}}\\
   &\quad+\frac{\vert x\vert }{\alpha}e^{-\frac{t}{\alpha}}\frac{C(1-e^{-t})^{(\gamma-d)/\alpha}}{((1-e^{-t})^{1/\alpha}+\vert y-e^{-\frac{t}{\alpha}}x\vert )^{\alpha+\ga}}\\
   &\le \frac{C(1+\frac{\vert x\vert }{\alpha}e^{-\frac{t}{\alpha}})(1-e^{-t})^{(\gamma-d)/\alpha}}{((1-e^{-t})^{1/\alpha}+\vert y-e^{-\frac{t}{\alpha}}x\vert )^{\alpha+\ga}},
\end{align*}
where the first inequality above follows from \cite[Theorem 1.2]{B-S-3}.
Thus, for $t>0$, $s>0$ small enough such that $(1-e^{-s/\alpha})\vert x\vert \le \frac{1}{2}(e^t-1)^{1/\alpha},$
\begin{align*}
  \vert q(t+s,x,y)-q(t,x,y)\vert  & \le s \frac{C(1+\frac{\vert x\vert }{\alpha}e^{-\frac{t}{\alpha}})(1-e^{-t})^{(\gamma-d)/\alpha}}{((1-e^{-t})^{1/\alpha}+\vert y-e^{-\frac{t}{\alpha}}x\vert )^{\alpha+\ga}}.
\end{align*}
In addition, according to (\ref{e:operator_OU}) and (\ref{e:OU}), we have
\begin{align}\label{heat}
\partial_{t}q(t,x,y)=\mathcal{A}^{\alpha,\nu}q(t,x,y).
\end{align}
Hence, since $\alpha+\gamma>d$, using dominated convergence theorem, (\ref{heat}) and Fubini's theorem, we have
\begin{align*}
\partial_{t} Q_{t}h(x)&=\partial_{t}\int_{\mathbb{R}^{d}}q(t,x,y)h(y)dy=\int_{\mathbb{R}^{d}}\partial_{t}q(t,x,y)h(y)dy\\
&=\int_{\mathbb{R}^{d}}\mathcal{A}^{\alpha,\nu}q(t,x,y)h(y)dy=\mathcal{A}^{\alpha,\nu}\int_{\mathbb{R}^{d}}q(t,x,y)h(y)dy=\mathcal{A}^{\alpha,\nu}Q_{t}h(x),
\end{align*}
the desired conclusion follows.
\end{proof}
\noindent

{\it Proof of Proposition \ref{p:solution}}.
From Remark \ref{rem} (i), we know that $f$ is well defined. Observing
\begin{align}\label{rep}
Q_{t}h(x)=\int_{\mathbb{R}^{d}}p_{1-e^{-t}}(y-e^{-\frac{t}{\alpha}}x)h(y)dy=\int_{\mathbb{R}^{d}}p_{1}(y)h\big((1-e^{-t})^{\frac{1}{\alpha}}y+e^{-\frac{t}{\alpha}}x\big)dy.
\end{align}
When $\alpha\in(1,2)$, since $h\in\mathrm{Lip}_1$
\begin{align*}
\Big\vert h\big((1-e^{-t})^{\frac{1}{\alpha}}y+e^{-\frac{t}{\alpha}}(x+z)\big)-h\big((1-e^{-t})^{\frac{1}{\alpha}}y+e^{-\frac{t}{\alpha}}x\big)\Big\vert \leq e^{-\frac{t}{\alpha}}\vert z\vert ,
\end{align*}
which immediately implies
\begin{align}\label{ineq2}
\Big\vert Q_{t}h(x+z)-Q_{t}h(x)\Big\vert \leq \int_{\mathbb{R}^{d}}p_{1}(y)e^{-\frac{t}{\alpha}}\vert z\vert dy=e^{-\frac{t}{\alpha}}\vert z\vert .
\end{align}
Recall $\mathcal{A}^{\alpha,\nu}f(x)=\mathcal{L}^{\alpha,\nu}f(x)-\frac{1}{\alpha}\langle x,\nabla f(x)\rangle.$
By \eqref{ineq2}, using the dominated convergence theorem, we get that
$$\nabla f(x)=-\int_0^\infty \nabla Q_th(x)\,dt.$$
Furthermore, we have by (\ref{fg4})
\begin{align*}
\big\vert \nabla_{x} p_{1-e^{-t}}(y-e^{-\frac{t}{\alpha}}x)\big\vert =&\big\vert e^{-\frac{t}{\alpha}}\nabla_{y} p_{1-e^{-t},\beta}(y-e^{-\frac{t}{\alpha}}x)\big\vert\\ \leq&\frac{C(e^t-1)^{-1/\alpha}(1-e^{-t})}{((1-e^{-t})^{1/\alpha}+\vert y-e^{-\frac{t}{\alpha}}x\vert )^{\alpha+\ga}},
\end{align*}
then for $x\in\mathbb{R}^{d},$ $z\in\mathbb{R}^{d}$ such that $\vert z\vert \leq\frac{1}{2}(e^{t}-1)^{\frac{1}{\alpha}},$
\begin{align*}
&\big\vert p_{1-e^{-t}}\big(y-e^{-\frac{t}{\alpha}}(x+z)\big)-p_{1-e^{-t}}(y-e^{-\frac{t}{\alpha}}x)\big\vert\\
\leq&\vert z\vert \frac{C2^{d+\alpha}(e^t-1)^{-1/\alpha}(1-e^{-t})}{((1-e^{-t})^{1/\alpha}+\vert y-e^{-\frac{t}{\alpha}}x\vert )^{\alpha+\ga}}.
\end{align*}
Hence, one can derive from the dominated convergence theorem and integration by parts that
\begin{align*}
\left\vert \nabla_{x}Q_{t}\big(h(x)-\pi(h)\big)\right\vert &=\left\vert \int_{\mathbb{R}^{d}}\nabla_{x}p_{1-e^{-t}}(y-e^{-\frac{t}{\alpha}}x)\big(h(y)-\pi(h)\big)dy\right\vert \nonumber\\
&=\left\vert e^{-\frac{t}{\alpha}}\int_{\mathbb{R}^{d}}\nabla_{y}p_{1-e^{-t}}(y-e^{-\frac{t}{\alpha}}x)\big(h(y)-\pi(h)\big)dy\right\vert \nonumber\\
&=\left\vert e^{-\frac{t}{\alpha}}\int_{\mathbb{R}^{d}}p_{1-e^{-t}}(y-e^{-\frac{t}{\alpha}}x)\nabla h(y)dy\right\vert \leq\| \nabla h\| _{\infty}e^{-\frac{t}{\alpha}},
\end{align*}
and similarly by \eqref{fg4}
\begin{align*}
\left\vert \nabla^{2}_{x}Q_{t}\big(h(x)-\pi(h)\big)\right\vert =&\left\vert e^{-\frac{2t}{\alpha}}\int_{\mathbb{R}^{d}}\nabla_{y}p_{(1-e^{-t})^{\frac{1}{\alpha}}}(y-e^{-\frac{t}{\alpha}}x)\cdot\nabla h(y)^{T})dy\right\vert \\
\leq&(1-e^{-t})^{-\frac{1}{\alpha}}e^{-\frac{2t}{\alpha}}\leq t^{-\frac{1}{\alpha}}e^{-\frac{t}{\alpha}}.
\end{align*}
These imply
\begin{align*}
&\big\vert \mathcal{L}^{\alpha,\nu}f(x)\big\vert\\ 
\leq&d_{\alpha}\int_{\mathbb{S}^{d-1}}\int_{0}^{\infty}\int_{0}^{\infty}\frac{\big\vert Q_{t}h(x+r\theta)-Q_{t}h(x)-\langle r\theta,\nabla Q_{t}h(x)\rangle\big\vert }{r^{\alpha+1}}dtdr\nu(d\theta)\\
\leq&d_{\alpha}\int_{\mathbb{S}^{d-1}}\int_{0}^{1}\int_{0}^{\infty}\int_{0}^{1}\int_{0}^{1}\frac{sr^{2}\big\vert \nabla^{2}Q_{t}h(x+sur\theta)\big\vert }{r^{\alpha+1}}dudsdtdr\nu(d\theta)\\
&+d_{\alpha}\int_{\mathbb{S}^{d-1}}\int_{1}^{\infty}\int_{0}^{\infty}\int_{0}^{1}\frac{r\big\vert \nabla Q_{t}h(x+sr\theta)-\nabla Q_{t}h(x)\big\vert }{r^{\alpha+1}}dsdtdr\nu(d\theta)\leq C.
\end{align*}

When $\alpha\in(0,1]$, since $h\in\mathcal H_\beta$
\begin{align*}
\Big\vert h\big((1-e^{-t})^{\frac{1}{\alpha}}y+e^{-\frac{t}{\alpha}}(x+z)\big)-h\big((1-e^{-t})^{\frac{1}{\alpha}}y+e^{-\frac{t}{\alpha}}x\big)\Big\vert \leq e^{-\frac{\beta t}{\alpha}}(\vert z\vert \wedge\vert z\vert ^{\beta}).
\end{align*}
By (\ref{rep}), we immediately have
\begin{align}\label{ineq}
\Big\vert Q_{t}h(x+z)-Q_{t}h(x)\Big\vert \leq \int_{\mathbb{R}^{d}}p(1,y)e^{-\frac{\beta t}{\alpha}}(\vert z\vert \wedge\vert z\vert ^{\beta})dy=e^{-\frac{\beta t}{\alpha}}(\vert z\vert \wedge\vert z\vert ^{\beta}).
\end{align}
Recall $\mathcal{A}^{\alpha,\nu}f(x)=\mathcal{L}^{\alpha,\nu}f(x)-\frac{1}{\alpha}\langle x,\nabla f(x)\rangle.$
By \eqref{ineq}, using the dominated convergence theorem, we get that
$$\nabla f(x)=-\int_0^\infty\nabla Q_th(x)\,dt.$$
Furthermore, if $\alpha=1$, we have
\begin{align*}
&\mathcal{L}^{1,\nu}f(x)\\
=&-d_{1}\int_{\mathbb{S}^{d-1}}\!\int_{0}^{\infty}\!\!\!\int_{0}^{\infty}\!\frac{Q_{t}h(x+r\theta)-Q_{t}h(x)-\langle r\theta,\nabla Q_{t}h(x){\bf 1}_{(0,1)}(r)\rangle}{r^{2}}dtdr\nu(d\theta)\\
=&-d_{1}\int_{\mathbb{S}^{d-1}}\int_{0}^{1}\int_{0}^{\infty}\int_{0}^{1}\frac{\langle r\theta,\nabla Q_{t}h(x+sr\theta)\rangle-\langle r\theta,\nabla Q_{t}h(x)\rangle}{r^{2}}dsdtdr\nu(d\theta)\\
&+-d_{1}\int_{\mathbb{S}^{d-1}}\int_{1}^{\infty}\int_{0}^{\infty}\frac{Q_{t}h(x+r\theta)-Q_{t}h(x)}{r^{2}}dtdr\nu(d\theta),
\end{align*}
and by integration by parts,
\begin{align*}
&\Big\vert \nabla Q_{t}h(x+zs)-\nabla Q_{t}h(x)\Big\vert \\
\leq& e^{-t}\int_{\mathbb{R}^{d}}\Big\vert p\big(y-(1-e^{-t})^{-1}e^{-t}(x+zs)\big)-p\big(y-(1-e^{-t})^{-1}e^{-t}x\big)\Big\vert \\
 &\qquad\qquad\qquad\qquad\qquad\qquad\qquad\qquad\qquad\qquad\quad\cdot\big\vert\nabla h\big((1-e^{-t})y\big)\big\vert dy\\
\leq&C e^{-t}\big((1-e^{-t})^{-1}e^{-t}\vert zs\vert \wedge1\big)\leq C e^{-t}\big(t^{-1}\vert zs\vert \wedge1\big),
\end{align*}
these imply
\begin{align*}
&\int_{\mathbb{S}^{d-1}}\int_{0}^{\infty}\int_{0}^{\infty}\frac{\big\vert Q_{t}h(x+r\theta)-Q_{t}h(x)-\langle r\theta,\nabla Q_{t}h(x){\bf 1}_{(0,1)}(r)\rangle\big\vert }{r^{2}}dtdr\nu(d\theta)\\
\leq&C\int_{\mathbb{S}^{d-1}}\int_{0}^{1}\int_{0}^{\infty}\frac{e^{-t}\big(t^{-1}r\wedge1\big)}{r}dtdr\nu(d\theta)
+\int_{\mathbb{S}^{d-1}}\int_{1}^{\infty}\int_{0}^{\infty}\frac{\e^{-\frac{\beta t}{\alpha}}r^{\beta}}{r^{2}}dtdr\nu(d\theta)\\
=&C\int_{0}^{1}\int_{0}^{\infty}\frac{e^{-t}\big(t^{-1}r\wedge1\big)}{r}dtdr
+\int_{1}^{\infty}\int_{0}^{\infty}\frac{e^{-\frac{\beta t}{\alpha}}r^{\beta}}{r^{2}}dtdr\leq C.
\end{align*}
If $\alpha\in(0,1),$ we have
\begin{align*}
\mathcal{L}^{\alpha,\nu}f(x)=-d_{\alpha}\int_{\mathbb{S}^{d-1}}\int_{0}^{\infty}\int_{0}^{\infty}\frac{Q_{t}h(x+r\theta)-Q_{t}h(x)}{r^{\alpha+1}}dtdr\nu(d\theta),
\end{align*}
and we have by (\ref{ineq})
\begin{align*}
&\int_{\mathbb{S}^{d-1}}\int_{0}^{\infty}\int_{0}^{\infty}\frac{\big\vert Q_{t}h(x+r\theta)-Q_{t}h(x)\big\vert }{r^{\alpha+1}}dtdr\nu(d\theta)\\
\leq&\int_{\mathbb{S}^{d-1}}\int_{0}^{\infty}\int_{0}^{\infty}\e^{-\frac{\beta t}{\alpha}}\frac{r\wedge r^{\beta}}{r^{\alpha+1}}dtdr\nu(d\theta)\\
=&\int_{0}^{\infty}\frac{r\wedge r^{\beta}}{r^{\alpha+1}}dr\int_{0}^{\infty}\e^{-\frac{\beta t}{\alpha}}dt\leq C.
\end{align*}

Therefore, by Fubini's theorem, we have
\begin{align*}
\mathcal{L}^{\alpha,\nu}f(x)=-\int_{0}^{\infty}\mathcal{L}^{\alpha,\nu}Q_{t}h(x)dt.
\end{align*}
Hence, according to Lemma \ref{A1}, we can obtain
\begin{align*}
\mathcal{A}^{\alpha,\nu}f=-\int_{0}^{\infty}\mathcal{A}^{\alpha,\nu}Q_{t}hdt=-\int_{0}^{\infty}\partial_{t}Q_{t}hdt=Q_{0}h-Q_{\infty}h,
\end{align*}
here $Q_{\infty}=\pi,$ the unique invariant distribution of the semigroup $(Q_t)_{t\ge 0}$ associated with $\mathcal{A}^{\alpha}$ by \cite[Cor. 17.9]{S}. The proof is complete.
\qed
\end{appendices}

\bigskip
{\bf Acknowledgements:} We warmly thank an anonymous referee for her/his very careful and detailed reading, constructive remarks and useful suggestions. P.\ Chen is supported by the NSF of Jiangsu Province grant BK20220867 and the Initial Scientific Research Fund of Young Teachers in Nanjing University of Aeronautics and Astronautics (1008-YAH21111). L. Xu is partially supported by NSFC No. 12071499, Macao S.A.R grant FDCT  0090/2019/A2 and University of Macau grant MYRG2020-00039-FST. 

\bigskip

{\bf Data Availability} Data sharing is not applicable to this article as no datasets were generated or analyzed during the current study. \\

{\bf Declarations} \\

{\bf Conflict of interest}

There are no competing interests to declare which arose during the preparation or publication process of this article.

\end{document}